\documentclass[12pt,showkeys]{amsart}

\usepackage{amsthm}
\usepackage{color}
\usepackage{graphicx}
\usepackage{subfigure}
\usepackage{amssymb}
\usepackage{amsmath}
\usepackage{multirow}
\usepackage{enumerate}
\usepackage{epstopdf}
\usepackage{cite,stmaryrd,txfonts}
\usepackage{tikz}
\usepackage{appendix}
\usetikzlibrary{snakes}
\usepackage{verbatim}
\usepackage{epic}
\usepackage{empheq}
\usepackage{ulem}

\input{undertilde}
\newcommand{\bs}{\boldsymbol}
\newcommand{\mrm}{\mathrm}

\def\Xint#1{\mathchoice
{\XXint\displaystyle\textstyle{#1}}%
{\XXint\textstyle\scriptstyle{#1}}%
{\XXint\scriptstyle\scriptscriptstyle{#1}}%
{\XXint\scriptscriptstyle\scriptscriptstyle{#1}}%
\!\int}
\def\XXint#1#2#3{{\setbox0=\hbox{$#1{#2#3}{\int}$ }
	\vcenter{\hbox{$#2#3$ }}\kern-.6\wd0}}

\def\dashint{\Xint-}
\def\usigma{\undertilde{\sigma}}
\def\utau{\undertilde{\tau}}
\def\uphi{\undertilde{\phi}}
\def\upsi{\undertilde{\psi}}

\def\uzeta{\undertilde{\zeta}}

\def\ugamma{\undertilde{\gamma}}

\def\uf{\undertilde{f}}

\def\uH{\undertilde{H}}

\def\un{\undertilde{n}}
\def\ur{\undertilde{r}}
\def\us{\undertilde{s}}

\def\ut{\undertilde{t}}

\def\up{\undertilde{p}}

\def\ux{\undertilde{x}}

\def\uL{\undertilde{L}}

\def\rot{{\rm rot}}
\def\curl{{\rm curl}}
\def\dv{{\rm div}}

\def\uH{\undertilde{H}}
\def\uL{\undertilde{L}}
\def\usigma{\undertilde{\sigma}}
\def\utau{\undertilde{\tau}}
\def\uphi{\undertilde{\phi}}
\def\upsi{\undertilde{\psi}}

\def\uzeta{\undertilde{\zeta}}

\def\ugamma{\undertilde{\gamma}}

\def\uf{\undertilde{f}}

\def\up{\undertilde{p}}

\def\rot{{\rm rot}}
\def\curl{{\rm curl}}
\def\dv{{\rm div}}

\newtheorem{theorem}{Theorem}
\newtheorem{remark}[theorem]{Remark}

\newtheorem{lemma}[theorem]{Lemma}

\newtheorem{definition}[theorem]{Definition}

\newcounter{mnote}
\let\oldmarginpar\marginpar
\renewcommand\marginpar[1]{\-\oldmarginpar[\raggedleft\footnotesize #1]%
{\raggedright\footnotesize #1}}

\begin{document}

\title[Polynomial de Rham complex on general quadrilateral grids]{Lowest-degree piecewise polynomial de Rham complex on general quadrilateral grids}

\author{Qimeng Quan, Xia Ji, Shuo Zhang}

%
%
\address{University of Chinese Academy of Sciences, Beijing 100190, People's Republic of China}
\email{quanqimeng@lsec.cc.ac.cn}
\address{LSEC, Institute of Computational Mathematics and Scientific/Engineering Computing, Academy of Mathematics and System Sciences, Chinese Academy of Sciences, Beijing 100190, People's Republic of China}
\email{jixia@lsec.cc.ac.cn}
\address{LSEC, Institute of Computational Mathematics and Scientific/Engineering Computing, Academy of Mathematics and System Sciences, Chinese Academy of Sciences, Beijing 100190, People's Republic of China}
\email{szhang@lsec.cc.ac.cn}

\thanks{The research of X. Ji is supported by the National Natural Science Foundation of China under grants 91630313 and 11971468, and National Centre for Mathematics and Interdisciplinary Sciences, CAS. The research of S. Zhang is supported partially by the National Natural Science Foundation of China with Grant Nos. 11471026 and 11871465 and National Centre for Mathematics and Interdisciplinary Sciences, CAS}



\begin{abstract} 
This paper is devoted to the construction of finite elements on grids that consist of general quadrilaterals not limited in parallelograms. Two finite elements defined as Ciarlet's triple are established for the $H^1$ and $H(\rm rot)$ elliptic problems, respectively. An $\mathcal{O}(h)$ order convergence rate in energy norm for both of them and an $\mathcal{O}(h^2)$ order convergence in $L^2$ norm for the $H^1$ scheme are proved under the asymptotic-parallelogram assumption on the grids. Further, the two finite element spaces on general quadrilateral grids, together with the space of piecewise constant functions, formulate a discretized de Rham complex.

The finite element spaces consist of piecewise polynomial functions, and, thus, are nonconforming on general quadrilateral grids. Indeed, a rigorous analysis is given in this paper that it is impossible to construct a practically useful finite element defined as Ciarlet's triple that can formulate a finite element space which consists of continuous piecewise polynomial functions on a grid that may include arbitrary quadrilaterals.
\end{abstract}

\maketitle
%
%
\section{Introduction}
\label{intro}

There has been a long history on the study of finite element methods on general quadrilateral grids. Many conforming and nonconforming finite elements have been established for various model problems. A classical strategy for constructing quadrilateral elements is to utilize isoparametric technique (cf., e.g., \cite{ciarlet2002finite}). With this strategy, one begins with a given shape function space on a reference square and a class of bilinear transforms (or Piola transforms, etc.), the finite element on any convex quadrilateral cell can be constructed correspondingly. Great success has been achieved via this approach, particularly in constructing conforming finite element spaces; we refer to \cite{arnold2002approximation,arnold2005quadrilateral,falk2011hexahedral} for more details. On the other hand, a solid difficulty of these methods, as discussed in, e.g., Zhang \cite{zhang2004nested}, is that one will encounter the problem of rational function integration in practical numerical computation due to non-constant Jacobian determinants and inverse Jacobian matrices. The same difficulty may happen for theoretical issues. We take the construction of discretized differential complexes for example. This is one of the most fundamental structural feature of the finite element schemes and has been a central topic of the finite element methods during the passed decades. If we consider the fundamental de Rham complex which reads in two dimension
\begin{equation}\label{eq:dr2d}
\mathbb{R}\xrightarrow{\rm inclusion} H^1  \xrightarrow{\nabla}  H(\rot) \xrightarrow{\rot} L^2,
\end{equation}
we have to pay extra attention to the different but relevant Jacobian matrices for the spaces, respectively. We are thus motivated to study finite element schemes with piecewise polynomials on every cell which possess clearer structure and more friendly implementation.
~\\

There have been various finite element schemes on general quadrilateral grids with piecewise polynomials; we refer to, e.g., \cite{rannacher1992simple,jeon2013class,park2003p,zhang2016stable} for details. It is worthy of attention that all these aforementioned finite element spaces are relatively nonconforming ones. Based on this observation, we prove in Section \ref{sec:geo} that, roughly speaking, it is impossible to construct a practically useful finite element whose shape functions are always piecewise polynomials and which can form conforming subspaces on a grid that consists of arbitrary quadrilaterals rather than parallelograms only. We thus do not seek to construct conforming elements in this paper; instead, hinted by the construction of discretized de Rham complex, we seek to construct so-called quasi-conforming ones. Namely, as to the finite element for $H^1$ problem, continuity is imposed on the vertices of the grids, and as to that for $H(\rot)$ problem, continuity is imposed on the tangential average along the interior edges. With respect to continuity restriction of such type, two lowest-degree finite elements are constructed on arbitrary quadrilateral grids. Together with the space of piecewise constants, the two newly constructed finite element spaces formulate a discretized de Rham complex.
~\\

Different from most existing nonconforming elements, the newly designed finite element for the $H^1$ problem does not pass the patch test on general quadrilateral grids. As it is conforming on parallelogram patches and recovers the bilinear element on rectangular patches, its moduli of continuity can be of $\mathcal{O}(h)$ order on a $\mathcal{O}(h^2)$ asymptotic-parallelogram grid. We note that this $\mathcal{O}(h^2)$ asymptotic-parallelogram assumption is quite frequently used, if not standard, in the analysis of the nonconforming finite elements. Further, the finite element space does not contain a nontrivial conforming subspace, and the standard duality arguments (c.f., e.g.,\cite{brenner2007mathematical,shi1984convergence}) can not be directly used. With the help of the specific commutative diagram and stability of the interpolator we establish in the sequel, however, we finally manage to prove that $\mathcal{O}(h^2)$ order in $L^2$ norm hinted by the similar technique in \cite{zeng2019optimal}. We have to point out that, though the $H(\rot)$ element designed is also nonconforming, different from existing ones for which higher regularity assumption than general is assumed for a same convergence rate (cf., e.g., \cite{shi2009low}), by the aid of this $\mathcal{O}(h^2)$ asymptotic-parallelogram assumption, we can prove the optimal convergence rate in energy norm for the proposed elements with the proper regularity assumptions of the exact solutions.
~\\

This paper is relevant to but different from some existing works on the discussed topic. For example, discretized de Rham complex was studied in \cite{arbogast2018direct}, where however rational functions are used together with polynomials for conformity; similarly, both polynomials and rational functions are used in \cite{arbogast2016two} on quadrilateral grids; \cite{bochev2008rehabilitation} saved the loss of convergence for $H(\dv)$ on non-affine quadrilateral grids but it is based on a mimetic divergence operator rather than an original one; \cite{dubach2009pseudo} studied finite elements that can pass the patch test, which can not be fulfilled by this lowest-degree finite element. We finally remark that, to our best knowledge, the first finite element complex with piecewise polynomials on general quadrilateral grids can be found in \cite{zhang2016stable}. The finite elements in \cite{zhang2016stable} are in some sense of the lowest degree for the Stokes complex, but they can pass the patch test. A comparision between the complexes in \cite{zhang2016stable} and in the present paper can thus be an interesting topic.
~\\

The remaining of the paper is organized as follows. In the remaining of this section, we introduce some necessary notations. In Section \ref{sec:geo}, we introduce some geometrical features of the grids. We will particularly prove that, shortly speaking, no practically useful conforming element with piecewise polynomials can be constructed for general quadrilateral grids. In Section \ref{sec:a sequence FEs}, three finite elements are introduced with a commutative diagram. The approximation error is analyzed by the technique of combining the classical Taylor expansion procedure and a specific commutative property. In Section \ref{sec:NFEs and mc}, the modulus of the continuity of the finite element functions are given. Then in Section \ref{sec:FE schemes for model problems}, the performance of the finite elements are studied for the $H^1$ and $H(\rot)$ problems; both theoretical analysis and numerical verifications are given. Some conclusions and comments are given in Section \ref{sec:concluding remarks}.
~\\

\paragraph{\bf Notations}

In this paper, conventional notations for the Sobolev spaces and grid-related quantities will be used. Let $\Omega\subset \mathbb{R}^2$ be a simple connected Lipschitz domain and $\Gamma = \partial\Omega$ be the piecewise boundary with $\un$ the outward unit normal vector and $\ut$ the counterclockwise unit tangential vector, ``$\undertilde{\cdot}$" representing the vector valued quantities. Denote by
$H^{m}(\Omega)$ and
$H^{m}_{0}(\Omega)$
the standard Sobolev spaces equipped with the norm $\|\cdot\|_{m,\Omega}$ and seminorm $|\cdot|_{m,\Omega}$ as usual, and $L^2_0(\Omega):=\{q\in L^2(\Omega):\int_{\Omega} q \ \mathrm{d}x=0\}$. We also denote by $\uL^2(\Omega)=(L^{2}(\Omega))^2$ , $\uH^m(\Omega)=(H^{m}(\Omega))^2$ and $\uH{}_{0}^{m}(\Omega)=(H^{m}_0(\Omega))^2$.
The inner product of $L^2$ and $\uL^2$ is denoted by $(\cdot,\cdot)$ on the domain $\Omega$.
We define two forms of rotation operator in two-dimensional case by
\begin{align*}
& \text{Given a vector}\
\usigma(x_{1},x_{2})
=(\sigma_{1},\sigma_{2})^{T}    &
\mathrm{rot}\usigma&=
\frac{\partial \sigma_{2}}{\partial x_{1}}-
\frac{\partial \sigma_{1}}{\partial x_{2}} \\
& \text{Given a scalar function}\
\sigma=\sigma(x_{1},x_{2})      &
\undertilde{\mathrm{curl}}\sigma&=
(\frac{\partial\sigma}{\partial x_{2}},
-\frac{\partial\sigma}{\partial x_{1}})^{T}.
\end{align*}
Superscript $T$ indicates transposition of vector or matrix as usual.
We also use these notations to denote Sobolev spaces
$H(\mathrm{rot},\Omega)=
\{\usigma|\usigma\in \uL^2(\Omega),\mathrm{rot}\usigma\in L^{2}(\Omega)\}
$
and
$
H_{0}(\mathrm{rot},\Omega)=\{\usigma|\usigma\in H(\mathrm{rot},\Omega), \usigma\cdot\ut=0 \ \text{on}\ \Gamma\}$ equipped with the norm   $\|\usigma\|_{\rot,\Omega}=(\|\usigma\|^2_{0,\Omega}+\|\rot\usigma\|^2_{0,\Omega})^{\frac{1}{2}}$.
Specially, a new notation is used for the space
$
H^{1}(\mathrm{rot},\Omega)
\triangleq
\{
\usigma|\usigma\in \uH^1(\Omega), \mathrm{rot}\usigma\in H^{1}(\Omega)
\}$.

Let $\mathcal{J}_{h}$ be a regular subdivision of domain $\Omega$,
with the elements being convex quadrilaterals, i.e.,
$\Omega=\cup_{K\in\mathcal{J}_{h}} K$.
And any two distinct quadrilaterals $K_{1}$ and $K_{2}$ in $\mathcal{J}_{h}$
with $\bar{K_{1}}\cap \bar{{K}_{2}}\neq\emptyset$, share exactly one vertex or
have one edge in common.
Denote a finite element $(K,P_K,D_K)$ by Ciarlet's triple \cite{ciarlet2002finite}, the subscription $K$ implying the dependence of the quadrilateral $K$.
Let $\mathcal{N}_{h}$ denote the set of all the vertexes,
$\mathcal{N}_{h}=
\mathcal{N}_{h}^{i}\cup
\mathcal{N}_{h}^{b}$,
with
$\mathcal{N}_{h}^{i}$ and
$\mathcal{N}_{h}^{b}$
consisting of the interior vertexes and the boundary vertexes, respectively.
Similarly,
let
$\mathcal{E}_{h}=
\mathcal{E}_{h}^{i}\cup
\mathcal{E}_{h}^{b}$
denote the set of all the edges,
with
$\mathcal{E}_{h}^{i}$ and
$\mathcal{E}_{h}^{b}$
consisting of the interior edges and the boundary edges, respectively.
The subscript $h$ in various notations implies the dependence of the subdivision.
Denote by $h_{K}$  the diameter of each quadrilateral $K$
and the grid size
$h\triangleq\max_{K\in\mathcal{J}_{h}}{h_{K}}$.
On the edge $e$, we use $[\![]\!]_{e}$ for the jump across $e$.

Throughout the paper we denote by $C$ a positive constant not necessarily the same at each occurrence but always independent of the diameter $h_K$ or the grid size $h$. Denote $\lambda_{F,G}$ by the generalized eigenvalue of matrix pair $(F,G)$, i.e., $F\ux=\lambda_{F,G}G\ux$.
We use notations $P_e$ and $P_{K}$
to denote the average of the integral
on the edge $e$ and quadrilateral $K$, respectively.

\section{Geometry of the quadrilaterals}
\label{sec:geo}

\subsection{Quadrilateral and functions}
\label{subsec:quad and fun}

Let $K$ be a convex quadrilateral
with $A_{i}$ the vertices and $e_{i}$ the edges, $i = 1:4$,
see Figure \ref{fig:convquad}.
Let $m_{i}$ be the mid-point of $e_{i}$,
then the quadrilateral
$\square m_{1}m_{2}m_{3}m_{4}$
is a parallelogram.
The cross point of $m_{1}m_{3}$ and $m_{2}m_{4}$,
which is labelled as $O$,
is the midpoint of both
$m_{1}m_{3}$ and $m_{2}m_{4}$.
Denote
$\undertilde{r}=\overrightarrow{Om_4}$
and
$\undertilde{s}=\overrightarrow{Om_1}$.
Then, the coordinates of the vertices in the coordinate system $\undertilde{r}O\undertilde{s}$
are
$A_{1}(1+\alpha, 1+\beta)$,
$A_{2}(-1-\alpha, 1-\beta)$,
$A_{3}(-1+\alpha, -1+\beta)$,
$A_{4}(1-\alpha, -1-\beta)$
and for some $\alpha, \beta$.
Since $K$ is convex,
$|\alpha|+|\beta|<1$.
Without loss of generality,
we assume that
$\alpha>0$, $\beta>0$
and
$\undertilde{r}\times\undertilde{s}>0$. Here and after, we call $\alpha$, $\beta$ local shape parameters.

Define the shape regularity indicator of the quadrilateral $K$ by
$\mathcal{R}_K=\max\{\frac{|r||s|}{\ur\times\us},\frac{|r|}{|s|},\frac{|s|}{|r|}\}$. Evidently $\mathcal{R}_K\geqslant 1$, and $\mathcal{R}_K=1$ if and only if $K$ is a square. A given family of quadrilateral subdivisions $\{\mathcal{J}_h\}$ of $\Omega$ is called regular, if all the shape regularity indicators of the quadrilaterals of all the subdivisions are uniformly bounded.

\begin{figure}[htbp]
	\vspace{-0.7cm}
	\centering
	\begin{tikzpicture}
	\draw (0.05, 0.26)node{$O$};
	
	\draw[dashed](-2, 2)--(3, 2);
	\draw[dashed](-3.5, -2)--(1.5, -2);
	\draw[dashed](-3.5, -2)--(-2, 2);
	\draw[dashed](1.5, -2)--(3, 2);
	
	\draw[dashed](0.5, 2)--(-1, -2);
	\draw[dashed](2.25, 0)--(-2.75, 0);
	
	\draw(-2.345, 0.275)node{\Large$m_2$};
	\draw(-0.75, -1.9)node{\Large$m_3$};
	\draw(2.1, 0.35)node{\Large$m_4$};
	\draw(0.85, 1.75)node{\Large$m_1$};
	
	\draw(-3.4, 0.175)node{\Large$e_2$};
	\draw(-0.95, -2.4)node{\Large$e_3$};
	\draw(2.7, 0.05)node{\Large$e_4$};
	\draw(0.4, 2.35)node{\Large$e_1$};
	
	\draw[->, line width=3pt, -latex](-0.25, 0)-- node[auto] {$\undertilde{r}$} (2.25, 0);
	\draw[->, line width=3pt, -latex](-0.25, 0)-- node[auto] {$\undertilde{s}$} (0.5, 2);
	
	\draw(3.725, 2.6)--(-2.725, 1.4);
	\draw(3.725, 2.6)--(0.775, -2.6);
	\draw(-2.775, -1.4)--(0.775, -2.6);
	\draw(-2.775, -1.4)-- (-2.725, 1.4);
	
	\draw(-2.9, 1.7)node{\Large$A_{2}$};
	\draw(-2.8, -1.7)node{\Large$A_{3}$};
	\draw(0.9, -2.85)node{\Large$A_{4}$};
	\draw(3.85, 2.8)node{\Large$A_{1}$};
	
	\draw[dashed](0.5, 2)--(2.25, 0);
	\draw[dashed](2.25, 0)--(-1, -2);
	\draw[dashed](-1, -2)--(-2.75, 0);
	\draw[dashed](-2.75, 0)--(0.5, 2);
	
	\end{tikzpicture}
	
	\caption{Illustration of a convex quadrilateral $K$.}
	\label{fig:convquad}
	\vspace{-0.5cm}
\end{figure}
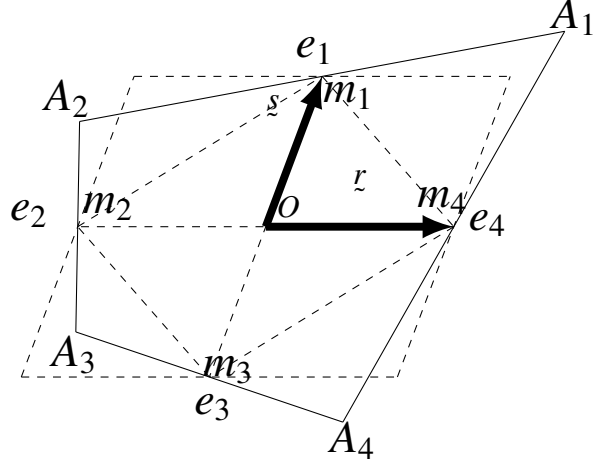

Define two linear functions $\xi$ and $\eta$ by $\xi(a\undertilde{r}+b\undertilde{s})=a$ and $\eta(a\undertilde{r}+b\undertilde{s})=b$.
The two functions play the same role on quadrilaterals as that played by barycentric coordinates on triangles.
Additionally we also define two functions $\hat{\xi}$ and $\hat{\eta}$ by $\hat{\xi}=\xi-\dashint_{K}\xi\ \mathrm{d}x$,
$\hat{\eta}=\eta-\dashint_{K}\eta\ \mathrm{d}x$
for convenience of calculation in Section \ref{sec:NFEs and mc}. Technically, we construct two tables \ref{tab:tab1}-\ref{tab:tab2} about the evaluation of some functions which will be useful in theoretical analysis and numerical computation.

\begin{table}[htbp]
	\vspace{-0.5cm}
	\caption{boundary integral evaluation of some functions}\label{tab:tab1}
	\begin{center}
		\begin{tabular}{|c|ccc|}
			\hline
			Function &
			$\xi^{2}$ &
			$\xi\eta$ &
			$\eta^{2}$ \\ \hline
			$\int_{e_{1}}$ &
			$\frac{(1+\alpha)^{2}|e_{1}|}{3}$ & $\frac{(1+\alpha)\beta|e_{1}|}{3}$ & $\frac{(3+\beta^{2})|e_{1}|}{3}$\\ \hline
			$\int_{e_{2}}$ &
			$\frac{(3+\alpha^{2})|e_{2}|}{3}$ & $\frac{\alpha(-1+\beta)|e_{2}|}{3}$ &
			$\frac{(1-\beta)^{2}|e_{2}|}{3}$\\ \hline
			$\int_{e_{3}}$ &
			$\frac{(1-\alpha)^{2}|e_{3}|}{3}$ & $\frac{(-1+\alpha)\beta|e_{3}|}{3}$ & $\frac{(3+\beta^{2})|e_{3}|)}{3}$\\ \hline
			$\int_{e_{4}}$ &
			$\frac{(3+\alpha^{2})|e_{4}|}{3}$ & $\frac{\alpha(1+\beta)|e_{4}|}{3}$ & $\frac{(1+\beta)^{2}|e_{4}|}{3}$\\ \hline
		\end{tabular}
	\end{center}
	\vspace{-1cm}
\end{table}

\begin{table}[htbp]
	\caption{integral evaluation of some functions in domain K}\label{tab:tab2}
	\begin{center}
		\begin{tabular}{|c|cccccc|}
			\hline
			Function & $1$ & $\xi$ & $\eta$ & $\xi^{2}$ & $\xi\eta$ & $\eta^{2}$ \\ \hline
			$\int_{K}$ &
			$4\ur\times\us$ & $\frac{4\beta}{3}\ur\times\us$ & $\frac{4\alpha}{3}\ur\times\us$ & $\frac{4}{3}(1+\alpha^{2})\ur\times\us$ & $\frac{4}{3}\alpha\beta\ur\times\us$& $\frac{4}{3}(1+\beta^{2})\ur\times\us$ \\ \hline
			Function & $1$ & $\hat{\xi}$ & $\hat{\eta}$ & $\hat{\xi}^{2}$ & $\hat{\xi}\hat{\eta}$ & $\hat{\eta}^{2}$ \\ \hline
			$\int_{K}$ &
			$4\ur\times\us$ &
			$0$ & $0$ & $\frac{4}{9}(3+3\alpha^{2}-\beta^2)\ur\times\us$ &
			$\frac{8}{9}\alpha\beta\ur\times\us$ & $\frac{4}{9}(3+3\beta^{2}-\alpha^2)\ur\times\us$ \\ \hline
		\end{tabular}
	\end{center}
\end{table}

\subsection{Grid refinement}
\label{subsec:mesh refinement}

Denote by $d_{K}$ the distance between the midpoints of the diagonals of the quadrilateral $K$, then we introduce a lemma.

\begin{lemma}
	All refined quadrilaterals produced by a bisection scheme of grid subdivisions have the property $d_{K}=\mathcal{O}(h^{2}_{K})$.
	\label{lemma:mesh refinement}
\end{lemma}

\begin{proof}
	The proof can be found in {\cite{shi1984convergence,zhang2004polynomial}}.
\end{proof}

Here and after, we call the grid generated by the bisection scheme as an asymptotically parallelogram grid.
We notice that the quantity $\max_{K\in\mathcal{J}_{h}}\{\alpha_K,\beta_K\}$ is of order $\mathcal{O}(h)$ uniformly for asymptotically regular parallelogram grid by Lemma \ref{lemma:mesh refinement}. This proposition will be used frequently in the Section \ref{sec:NFEs and mc}.

\subsection{On the construction of conforming element with piecewise polynomials}
\label{subsec:mcpp on quads}

Hereby, we prove that, as rigorously presented in the below theorem, for general quadrilateral grids, no practically useful conforming elements can be constructed with piecewise polynomials.

\begin{lemma}\label{lemma:nocon}
	Let $K_1$ be a given quadrilateral with $D$ being one of its vertices. Let $q$ be a polynomial defned on $K_1$, such that $q$ vanishes along the two opposite edges of $D$. If for any patch $\omega_D$ that consists of quadrilaterals and is centered at $D$ (including $K_1$, see Figure \ref{fig:patch for no CFE space} for a reference), we can find a piecewise polynomial $r\in H^1_0(\omega_D)$, such that $r|_{K_1}=q$, then $q$ vanishes on the boundary of $K_1$.
\end{lemma}
\begin{proof}
	First we observe that, by the continuity of the finite element space generated from Ciarlet's triple by the continuity of nodal parameters and by the arbitrariness of choosing the evaluation of nodal parameters, for any groups of quadrilaterals (including $K_1$) which can form a patch $\omega_D$ centered at $D$, there exists a piecewise polynomial $r\in H^1_0(\omega_D)$, such that $r|_K=q$. We emphasize that the evaluation on these common nodal parameters does not inflect the evaluation of the finite element function on the edges which do not intersect with $K_1$, and this is why $r$ can be chosen in $H^1_0(\omega_D)$.
	
	Now we assume that $q$ is nontrivial along one edge $e\ni D$ of $K_1$ and are going to establish a contradiction. Since $r|_f=0$, we can rewrite $r|_{K_2}=r_{-1}\cdot l_f$, where $l_f$ is a first degree polynomial which vanishes on $f$ and $r_{-1}$ is a polynomial with one degree lower than $r$. Without loss of generality, we assume that $f$ is not parallel to $e$. By the continuity of $r$ on $e$, we can rewrite $q|_e=q_{-1}l_f|_e=q_{-1}(t-\theta)$, where $t$ is the length parameter of $e$, $q_{-1}$ is a polynomial on $e$ with one degree lower than $q$ and $\theta$ varies as the angle between $e$ and $f$ varies.  Recall that  $K_1$ and $q_{K_1}$ are fixed, but $f$ can be arbitrary. Change $f$ to another direction $f'$, by elementary calculation, it follows that $q|_e=q_{-2}(t-\theta)(t-\theta')$. This way, by repeating the procedure, we can see that $q|_e$ contains a polynomial factor with growing degree and thus can not be a nontrivial polynomial. This leads to a contradiction to the assumption that $q|_e\not\equiv 0$ and completes the proof.
\end{proof}
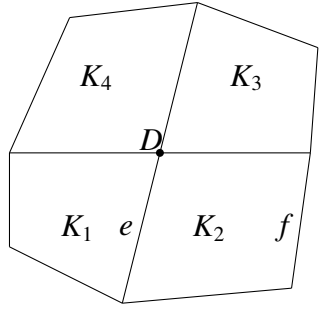
\begin{figure}[htbp]
	\centering
	\begin{tikzpicture}
	\draw(-1.5,-1.25)--(0,-2);
	\draw(0,-2)--(2.25,-1.8);
	\draw(-1.5,0)--(0.5,0);
	\draw(0.5, 0)--(2.5, 0);
	\draw(-0.7, 1.8)--(1,2);
	\draw(1,2)--(2.6, 1.4);
	
	\draw(-1.5,-1.25)--(-1.5,-0);
	\draw(-1.5,-0)--(-0.7, 1.8);
	\draw(0,-2)--(0.5,0);
	\draw(0.5, 0)--(1, 2);
	\draw(2.25,-1.8)--(2.5,-0);
	\draw(2.5, -0)--(2.6, 1.4);
	
	\draw(0.38,0.17)node{$D$};
	\draw(-0.6,-1)node{$K_{1}$};
	\draw(1.15,-1)node{$K_{2}$};
	\draw(-0.35,1)node{$K_{4}$};
	\draw(1.65,1)node{$K_{3}$};
	\draw(0.05,-1)node{$e$};
	\draw(2.15,-1)node{$f$};
	\fill(0.5,0) circle(1.5pt);
	\end{tikzpicture}
	\caption{The patch $\omega_D$ around $K_1$ centered at $D$.}
	\vspace{-0.5cm}
	\label{fig:patch for no CFE space}
\end{figure}

\begin{theorem}\label{thm:nocon}
	Let ${\rm FEM_{pq}}=(K, P_K, N_K)$ be a finite element defined by Ciarlet's triple, with $K$ being any quadrilateral, and $P_K$ being a space of polynomials on $K$. If the finite element space generated by ${\rm FEM_{pq}}$ by the continuity of nodal parameters is an $H^1$ subspace on any grid that consists of arbitrary quadrilaterals, then $P_K$ only contains polynomials that vanish on the boundary of $K$.
\end{theorem}
The proof follows immediately from Lemma \ref{lemma:nocon}.

\begin{remark}
	Shortly speaking, if a finite element ${\rm FEM_{pq}}$, the subscripts ``p" for {\it polynomial} and ``q" for {\it quadrilateral}, can formulate continuous piecewise polynomial space on general quadrilateral grids, then the shape function space of the finite element consists of bubble functions on $K$ only. This theorem shows the non-existence of practically useful conforming finite element defined by Ciarlet's triple on general quadrilateral grids. However, we emphasize that it does not exclude the possibility that, on a given quadrilateral grid, a subspace of $H^1_0(\Omega)$ that consists of piecewise $k$-th degree polynomials can contain more than cell bubbles.
\end{remark}

\begin{remark}\label{rem:nocon}
	Similarly, conforming finite elements can not be defined for $H(\rot)$ with piecewise polynomials for general quadrilateral grids. Indeed, the assertion can be generalized to general Sobolev spaces.
\end{remark}

\section{A sequence of lowest-degree finite elements}
\label{sec:a sequence FEs}

\subsection{Definitions of the finite elements}
\label{subsec:def of FEs}

In the subsection we introduce three types of finite elements.

The quadrilateral finite element presented below is similar to the bilinear element on rectangle, and we call it the quadrilateral bilinear (${\rm QBL}$) element.

\fbox{
	\begin{minipage}{0.99\textwidth}
		The ${\rm QBL}$ element is defined by
		$(K,P_{K}^{\rm QBL},D_{K}^{\rm QBL})$ with
		\begin{enumerate}
			\item[1.] $K$ is a convex quadrilateral with vertexes
			$A_{i}$, $i=1:4$
			\item[2.] $P_{K}^{\rm QBL}
			\triangleq
			\mathrm{span}\{1, \xi, \eta, \xi \eta \}$
			\item[3.] $D_{K}^{\rm QBL}
			\triangleq
			\{u(A_{i}),i=1:4\}$
			\quad
			\text{for any}
			$u\in H^{2}(K)$
			\label{def:(K,P,V)}	
		\end{enumerate}
	\end{minipage}
}
${\rm QBL}$ element defined above is unisolvent.
Indeed, define
\begin{equation}
\left\{
\begin{aligned}
\phi_{1} & =
\frac{\alpha+\beta-1}{4(\alpha^{2}+\beta^{2}-1)}\xi\eta+ \frac{(\beta-1)(-\alpha+\beta+1)}{4(\alpha^{2}+\beta^{2}-1)}\xi+  \frac{(\alpha-1)(\alpha-\beta+1)}{4(\alpha^{2}+\beta^{2}-1)}\eta+
\frac{-(\alpha-1)(\beta-1)(\alpha+\beta+1)}{4(\alpha^{2}+\beta^{2}-1)} \\
\phi_{2} & =
\frac{-\alpha+\beta+1}{4(\alpha^{2}+\beta^{2}-1)}\xi\eta+ \frac{-(\beta+1)(\alpha+\beta-1)}{4(\alpha^{2}+\beta^{2}-1)}\xi+  \frac{(\alpha-1)(\alpha+\beta+1)}{4(\alpha^{2}+\beta^{2}-1)}\eta+
\frac{(\alpha-1)(\beta+1)(\alpha-\beta+1)}{4(\alpha^{2}+\beta^{2}-1)}  \\
\phi_{3} & =
\frac{-(\alpha+\beta+1)}{4(\alpha^{2}+\beta^{2}-1)}\xi\eta+ \frac{(\beta+1)(\alpha-\beta+1)}{4(\alpha^{2}+\beta^{2}-1)}\xi+ \frac{(\alpha+1)(-\alpha+\beta+1)}{4(\alpha^{2}+\beta^{2}-1)}\eta+
\frac{(\alpha+1)(\beta+1)(\alpha+\beta-1)}{4(\alpha^{2}+\beta^{2}-1)}  \\
\phi_{4} & =
\frac{\alpha-\beta+1}{4(\alpha^{2}+\beta^{2}-1)}\xi\eta+ \frac{(\beta-1)(\alpha+\beta+1)}{4(\alpha^{2}+\beta^{2}-1)}\xi+ \frac{-(\alpha+ 1)(\alpha+\beta-1)}{4(\alpha^{2}+\beta^{2}-1)}\eta+
\frac{(\alpha+1)(\beta-1)(-\alpha+\beta+1)}{4(\alpha^{2}+\beta^{2}-1)}.
\end{aligned}
\right.
\label{equ:local dual basis of QBL}
\end{equation}
then we can verify directly
$\phi_{i}(A_{j})=\delta_{ij}$, $i,j=1:4$.

Here and after, the functions
$\{\phi_{i}\}_{i=1:4}$ are called local basis of
$P_{K}^{\rm QBL}$. We use the notation $\phi_i^{(j)}$ to denote $j$-th coefficient of $i$-th basis, for example,
$\phi_{1}^{(2)} = \frac{(\beta-1)(-\alpha+\beta+1)}{4(\alpha^{2}+\beta^{2}-1)}$.

Given a ${\rm QBL}$ element
$(K,P_{K}^{\rm QBL},D_{K}^{\rm QBL})$,
define the local interpolation operator $J_{K}$ by
\begin{equation}
J_{K}u=\sum_{i=1}^{4}u(A_{i})\phi_{i},
\quad
\forall
u \in H^2(K).
\label{def:J_K}
\end{equation}
Furthermore, given a family of ${\rm QBL}$ elements
$(K_{i},P_{K_{i}}^{\rm QBL},D_{K_{i}}^{\rm QBL})$ in a subdivision $\mathcal{J}_h$,
define the global interpolation operator $J_{h}$ by
\begin{equation}
J_{h}u|_{K_{i}}=
J_{K_{i}}u
\quad
\forall K_{i}\in\mathcal{J}_{h}.
\label{def:J_h}
\end{equation}

The quadrilateral finite element presented below is similar to the Raviart-Thomas element on rectangle, and we call it the quadrilateral Raviart-Thomas (${\rm QRT}$) element.

\fbox{
	\begin{minipage}{0.99\textwidth}
		The ${\rm QRT}$ element is defined by
		$(K,P_{K}^{\rm QRT},D_{K}^{\rm QRT})$ with
		\begin{enumerate}
			\item[1.]
			$K$ is a convex quadrilateral with edges
			$e_{i}$, $i=1:4$
			\item[2.] $P_{K}^{\rm QRT}
			\triangleq
			\mathrm{span}
			\{\nabla\xi,\nabla\eta,\xi\nabla\eta,\eta\nabla\xi\}$
			\item[3.] $D_{K}^{\rm QRT}
			\triangleq
			\{\dashint_{e_{i}}
			\undertilde{\sigma}
			\cdot
			\undertilde{t}{}_{i}\ \mathrm{d}s,i=1:4\}$
			\quad
			\text{for any}
			$\undertilde{\sigma}\in \uH^1(K)$
			\label{def:(K,P,D) of rot FEM}
		\end{enumerate}
		Here $\ut{}_{i}$ is the unit tangential vector of $e_{i}$ respectively and the positive direction is counterclockwise along $\partial K$.
	\end{minipage}
}

The ${\rm QRT}$ element as above is unisolvent.
Indeed, define
\begin{equation}
\left\{
\begin{aligned}
\uphi{}_{1} & =
\frac{(1-\alpha)(1-\beta^2)|e_1|}{4(\alpha^{2}+\beta^{2}-1)}\nabla\xi+ \frac{\alpha(1-\alpha)\beta|e_1|}{4(\alpha^{2}+\beta^{2}-1)}\nabla\eta+  \frac{-\alpha(1-\alpha)|e_1|}{4(\alpha^{2}+\beta^{2}-1)}\xi\nabla\eta+
\frac{(1-\alpha-\beta^2)|e_1|}{4(\alpha^{2}+\beta^{2}-1)}\eta\nabla\xi \\
\uphi{}_{2} & =
\frac{\alpha\beta(1+\beta)|e_2|}{4(\alpha^{2}+\beta^{2}-1)}\nabla\xi+ \frac{(1-\alpha^2)(1+\beta)|e_2|}{4(\alpha^{2}+\beta^{2}-1)}\nabla\eta+  \frac{-(1-\alpha^2+\beta)|e_2|}{4(\alpha^{2}+\beta^{2}-1)}\xi\nabla\eta+
\frac{-\beta(1+\beta)|e_2|}{4(\alpha^{2}+\beta^{2}-1)}\eta\nabla\xi  \\
\uphi{}_{3} & =
\frac{-(1+\alpha)(1-\beta^2)|e_3|}{4(\alpha^{2}+\beta^{2}-1)}\nabla\xi+ \frac{-\alpha(1+\alpha)\beta|e_3|}{4(\alpha^{2}+\beta^{2}-1)}\nabla\eta+ \frac{\alpha(1+\alpha)|e_3|}{4(\alpha^{2}+\beta^{2}-1)}\xi\nabla\eta+
\frac{(1+\alpha-\beta^2)|e_3|}{4(\alpha^{2}+\beta^{2}-1)}\eta\nabla\xi  \\
\uphi{}_{4} & =
\frac{-\alpha\beta(1-\beta)|e_4|}{4(\alpha^{2}+\beta^{2}-1)}\nabla\xi+ \frac{-(1-\alpha^2)(1-\beta)|e_4|}{4(\alpha^{2}+\beta^{2}-1)}\nabla\eta+ \frac{-(1-\alpha^2-\beta)|e_4|}{4(\alpha^{2}+\beta^{2}-1)}\xi\nabla\eta+
\frac{\beta(1-\beta)|e_4|}{4(\alpha^{2}+\beta^{2}-1)}\eta\nabla\xi.
\end{aligned}
\right.
\label{equ:local dual basis of QRT}
\end{equation}
Denote $\{D_i^{\rm QRT}\}_{i=1:4}$ by the components of $D_K^{\rm QRT}$, then we can verify directly $D_i^{\rm QRT}(\uphi{}_j)=\delta_{ij}$, $i,j=1:4$.

Here and after, the functions $\{\uphi{}_i\}_{i=1:4}$ are called local  basis of $P_K^{\rm QRT}$.

Given the ${\rm QRT}$ element
$(K,P_{K}^{\rm QRT},D_{K}^{\rm QRT})$,
define the local interpolation operator $\sqcap_{K}$ by
\begin{equation}
\sqcap_{K}
\undertilde{\sigma} =
\sum_{i=1}^{4}
D_i^{\rm QRT}(\usigma)\uphi{}_i
\quad
\forall
\usigma \in \uH^1(K).
\label{def:local interpolation of rot}
\end{equation}
Furthermore, given a family of ${\rm QRT}$ elements $(K_i,P_{K_i}^{\rm QRT},D_{K_i}^{\rm QRT})$ in a subdivision
$\mathcal{J}_{h}$,
define the global interpolation operator
$\sqcap_{h}$ by
\begin{equation}
\sqcap_{h}\undertilde{\sigma}|_{K_{i}}=
\sqcap_{K_{i}}\undertilde{\sigma}
\quad
\forall K_{i}\in\mathcal{J}_{h}.
\label{def:global rot interpolation}
\end{equation}

Finally, for any $q\in L^2(\Omega)$ define the interpolation operator $P_h$ by
$P_hq|_{K_i}= P_{K_{i}}q, \forall K_{i}\in \mathcal{J}_h$.

%
%
%
%
%
%
%

%
\subsubsection{Exact sequences on a quadrilateral}
\label{subsubsec:exact sequence on a quad}

\begin{theorem}
	The commutative diagram holds as below:
	\begin{equation*}
	\begin{array}{ccccccccc}
	\mathbb{R} & ~~~\longrightarrow~~~ & H^2(K) & ~~~\xrightarrow{\bs{\mrm{\nabla}}}~~~ & \uH^1(K) & ~~~\xrightarrow{\mrm{rot}}~~~ & L^2(K)  \\
	& & \downarrow J_K & & \downarrow \sqcap_{K} & & \downarrow P_K & & \\
	\mathbb{R} & \longrightarrow & P_K^{\rm QBL} & \xrightarrow{\bs{\mrm{\nabla}}} & P_K^{\rm QRT} & \xrightarrow{\mrm{rot}} & \mathbb{R}.
	\end{array}
	\end{equation*}
	\label{thm:commutativity on a cell}
\end{theorem}
\begin{proof}	
	We first prove the discretized de Rham complex. Evidently ${\rm ker}(\nabla)=\mathbb{R}$ and
	$\nabla P^{\rm QBL}_K\subset P_{K}^{\rm QRT}$.
	On the other hand, $\rot P_{K}^{\rm QRT}=\mathbb{R}$.
	It remains to prove that
	${\rm ker}({\rm rot})=\nabla P^{\rm QBL}_K$. Given a $\utau\in P_K^{\rm QRT}$, such that ${\rm rot}\utau=0$. Since $\utau= d_1\nabla\xi+d_2\nabla\eta+d_3\xi\nabla\eta+d_4\eta\nabla\xi$, then we have $d_3=d_4$ and $\utau \in \nabla P^{\rm QBL}_K$.
	
	Then we are going to show that $\nabla J_K=\sqcap_{K}\nabla$ on $H^2(K)$, and $\mathrm{rot}\sqcap_{K}=P_K\mathrm{rot}$ on $\uH^1(K)$. We first prove the former. Given a $\usigma \in \uH^1(K)$, let $\sqcap_{K}\usigma=g_{1}\cdot\nabla\xi+g_{2}\cdot\nabla\eta+g_{3}\cdot\xi\nabla\eta+g_{4}\cdot\eta\nabla\xi$. By definition, we have
	\begin{align*}
	g_{1}&=
	\frac{
		(1-\alpha)(1-\beta^2)\undertilde{e}{}_{1}\cdot P_{e_{1}}(\usigma)+
		\alpha\beta(1+\beta)\undertilde{e}{}_{2}\cdot P_{e_{2}}(\usigma)-
		(1+\alpha)(1-\beta^{2})\undertilde{e}{}_{3}\cdot P_{e_{3}}(\usigma)-
		\alpha\beta(1-\beta)\undertilde{e}{}_{4}\cdot P_{e_{4}}(\usigma)
	}
	{4(\alpha^2+\beta^{2}-1)}\\
	g_{2}&=
	\frac{
		\alpha(1-\alpha)\beta\undertilde{e}{}_{1}\cdot P_{e_{1}}(\usigma)+
		(1-\alpha^{2})(1+\beta)\undertilde{e}{}_{2}\cdot P_{e_{2}}(\usigma)-
		\alpha(1+\alpha)\beta\undertilde{e}{}_{3}\cdot P_{e_{3}}(\usigma)-
		(1-\alpha^{2})(1-\beta)\undertilde{e}{}_{4}\cdot P_{e_{4}}(\usigma)
	}
	{4(\alpha^2+\beta^{2}-1)}\\
	g_{3}&=
	\frac{
		-\alpha(1-\alpha)\undertilde{e}{}_{1}\cdot P_{e_{1}}(\usigma)
		-(1-\alpha^{2}+\beta)\undertilde{e}{}_{2}\cdot P_{e_{2}}(\usigma)
		+\alpha(1+\alpha)\undertilde{e}{}_{3}\cdot P_{e_{3}}(\usigma)
		-(1-\alpha^{2}-\beta)\undertilde{e}{}_{4}\cdot P_{e_{4}}(\usigma)
	}
	{4(\alpha^2+\beta^{2}-1)}
	\\
	g_{4}&=
	\frac{
		(1-\alpha-\beta^2)\undertilde{e}{}_{1}\cdot P_{e_{1}}(\usigma)
		-\beta(1+\beta)\undertilde{e}{}_{2}\cdot P_{e_{2}}(\usigma)
		+(1+\alpha-\beta^{2})\undertilde{e}{}_{3}\cdot P_{e_{3}}(\usigma)
		+\beta(1-\beta)\undertilde{e}{}_{4}\cdot P_{e_{4}}(\usigma)
	}
	{4(\alpha^2+\beta^{2}-1)}.\\
	\end{align*}
	Now given a $u\in H^2(K)$, we take $\usigma=\nabla u\in \uH^1(K)$ and the former follows by simple calculation.
	
	It remains to prove the latter.
	Since
	$\nabla\xi=\frac{s^{\bot}}{\ur\times\us}$,  $\nabla\eta=\frac{r^{\bot}}{\ur\times\us}$,
	then for $\usigma\in\uH^1(K)$ it holds
	\begin{equation*}
	\mathrm{rot}(\sqcap_{K}\usigma)=
	(g_{3}-g_{4})\nabla\xi\times\nabla\eta=
	\frac{g_3-g_4}{\ur\times\us}=
	\frac{1}{4\ur\times\us}\int_{\partial K}
	\undertilde{\sigma}
	\cdot
	\undertilde{t}\ \mathrm{d}s=
	P_K({\rm rot}\usigma).
	\end{equation*}
	The proof is completed.
\end{proof}

\subsection{Interpolation error estimation}
\label{subsec:interpolation error}

\subsubsection{Interpolation error estimations in $L^2$ norm}
\label{subsubsec:interpolation error in L2 norm}

\begin{theorem}
	Let $K$ be a convex quadrilateral, then it holds
	\begin{equation*}
	\| u-J_{K}u\|_{0, K}
	\leqslant
	Ch^{2}_{K}|u|_{2, K}
	\quad
	\forall u\in H^2(K).
	\end{equation*}
\end{theorem}
\begin{proof}
	By density, it suffices to consider $u\in C^{2}(\bar{K})$.
	Let $A$ be any point in the quadrilateral $K$ with vertexes $\{A_i\}_{i=1:4}$.
	Using Taylor expansion with integral remainder, we have
	\begin{equation*}
	u(A_{i})=u(A)+\nabla u(A)\cdot(A_{i}-A)+R_{i}(A),
	\quad
	R_{i}(A)=\int_{0}^{1}
	(1-t)
	\frac{\mathrm{d}^{2}u}
	{\mathrm{d} t^{2}}
	(\xi_{i},\eta_{i})
	\ \mathrm{d}t.
	\end{equation*}
	Here
	$\xi_{i}=tx_{i}+(1-t)x$,
	$\eta_{i}=ty_{i}+(1-t)y$.
	
	Since $J_{K}$ preserves linear polynomial
	\begin{equation*}
	J_{K}u(A)
	=\sum_{i=1}^{4}
	u(A_{i})\phi_{i}(A)
	=u(A)+\sum_{i=1}^{4}R_{i}(A)\phi_{i}(A).
	\end{equation*}
	Then we obtain
	\begin{equation*}
	u(A)-J_{K}u(A)=-\sum_{i=1}^{4}R_{i}(A)\phi_{i}(A).
	\end{equation*}
	
	Evidently
	\begin{align*}
	|R_{i}(A)|^{2}
	&=|\int_{0}^{1}(1-t)
	(\frac{\partial^{2}u}{\partial \xi_{i}^{2}}(x_{i}-x)^{2}+
	2\frac{\partial^{2}u}{\partial\xi_{i}\partial\eta_{i}}(x_{i}-x)(y_{i}-y)+\frac{\partial^{2}u}{\partial \eta_{i}^{2}}(y_{i}-y)^{2})
	\ \mathrm{d}t|^{2} \\
	&\leqslant
	4h_{K}^{4}\int_{0}^{1}(1-t)^{2}
	\sum_{|m|=2}
	|\partial^{m}u(\xi_{i}, \eta_{i})|^2
	\ \mathrm{d}t.
	\end{align*}
	\begin{equation*}
	\| u-J_{K}u\|_{0, K}^{2}
	\leqslant
	Ch_{K}^{4}\sum_{i=1}^{4}\sum_{|m|=2}
	\int_{0}^{1}(1-t)^2\int_{K}|\partial^{m}u(\xi_{i}, \eta_{i})|^2\ \mathrm{d}x\mathrm{d}y\mathrm{d}t.
	\end{equation*}
	Take integral variable substitution:
	$\mathrm{d}\xi_{i}=(1-t)\mathrm{d}x$,
	$\mathrm{d}\eta_{i}=(1-t)\mathrm{d}y$,
	then we have
	\begin{equation*}
	\| u-J_{K}u\|_{0, K}^{2}
	\leqslant
	Ch_{K}^{4}
	\sum_{i=1}^{4}\sum_{|m|=2}
	\int_{0}^{1}\int_{K}
	|\partial^{m}u(\xi_{i}, \eta_{i})|^2
	\ \mathrm{d}\xi_{i}\mathrm{d}\eta_{i}\mathrm{d}t
	=Ch_{K}^{4}|u|_{2, K}^2.
	\end{equation*}
	The proof is completed.
\end{proof}

\begin{theorem}
	Let $K$ be a convex quadrilateral,
	then it holds
	\begin{equation}
	\|\usigma-\sqcap_{K}\usigma\|_{0, K}
	\leqslant Ch_{K}|\usigma|_{1, K}
	\qquad\forall\usigma\in\uH^1(K).
	\label{inequ:L2 norm for local estimation in H(rot)}
	\end{equation}
	\label{thm:L2 norm for interpolation estimation in H(rot)}
\end{theorem}

We postpone the proof of the theorem after a technical lemma.
Define a new interpolation operator
$\sqcap_{Q}: \uH^1(K)\to \undertilde{\rm{span}}\{1,\xi,\eta,\xi^{2}-\eta^{2}\}$ by
$\dashint_{e_{i}}\usigma\ \mathrm{d}s=
\dashint_{e_{i}}\sqcap_{Q}\usigma\ \mathrm{d}s, i=1:4$.
Evidently $\sqcap_{Q}$ is well-defined.

\begin{lemma}
	The local interpolation operator $\sqcap_{K}$ is $H^{1}$ stable, namely
	\begin{equation*}
	|\sqcap_{K}\usigma|_{1, K}\leqslant C|\usigma|_{1, K}
	\qquad\forall \usigma\in\uH^1(K).
	\end{equation*}
	\label{lemma:H^{1} stable on interpolation associated with QRT}	
\end{lemma}
\begin{proof}
	Let
	$
	\sqcap_{Q}\usigma=
	\undertilde{d}{}_{1}\cdot 1+
	\undertilde{d}{}_{2}\cdot\xi+
	\undertilde{d}{}_{3}\cdot\eta+
	\undertilde{d}{}_{4}\cdot(\xi^{2}-\eta^{2})
	$, then by definition we have
	\begin{align*}
	\undertilde{d}{}_{1}&=
	\frac{\alpha^{2}-\beta^{2}+2}{8}P_{e_{1}}(\usigma)+
	\frac{-\alpha^{2}+\beta^{2}+2}{8}P_{e_{2}}(\usigma)+
	\frac{\alpha^{2}-\beta^{2}+2}{8}P_{e_{3}}(\usigma)+
	\frac{-\alpha^{2}+\beta^{2}+2}{8}P_{e_{4}}(\usigma)\\
	\undertilde{d}{}_{2}&=
	-\frac{\beta}{4}P_{e_{1}}(\usigma)+
	\frac{\beta-2}{4}P_{e_{2}}(\usigma)-
	\frac{\beta}{4}P_{e_{3}}(\usigma)+
	\frac{\beta+2}{4}P_{e_{4}}(\usigma)\\
	\undertilde{d}{}_{3}&=
	\frac{\alpha+2}{4}P_{e_{1}}(\usigma)-
	\frac{\alpha}{4}P_{e_{2}}(\usigma)+
	\frac{\alpha-2}{4}P_{e_{3}}(\usigma)-
	\frac{\alpha}{4}P_{e_{4}}(\usigma)\\
	\undertilde{d}{}_{4}&=
	-\frac{3}{8}P_{e_{1}}(\usigma)+
	\frac{3}{8}P_{e_{2}}(\usigma)-
	\frac{3}{8}P_{e_{3}}(\usigma)+
	\frac{3}{8}P_{e_{4}}(\usigma).
	\end{align*}
	
	Denote by row vector
	$\usigma=(\sigma_{1},\sigma_{2})$,
	$\undertilde{p}{}_{1}=(P_{e_{i}}(\sigma_{1}))_{i=1:4}$,
	$\undertilde{p}{}_{2}=(P_{e_{i}}(\sigma_{2}))_{i=1:4}$,
	and column vector
	$\undertilde{p}=
	(\undertilde{p}{}_{1},\undertilde{p}{}_{2})^{T}$.
	Denote $\undertilde{=}$ by $=$ up to a constant,
	independent of diameter $h_{K}$,
	then
	\begin{align*}
	\int_{K}|\nabla(\sqcap_{Q}\usigma)|^{2}\ \mathrm{d}x
	&\undertilde{=}
	h_{K}^{-2}\int_{K}
	|\partial_{\xi}(\sqcap_{Q}\usigma)|^{2}+
	|\partial_{\eta}(\sqcap_{Q}\usigma)|^{2}
	\ \mathrm{d}x
	=h_{K}^{-2}\int_{K}|\undertilde{d}{}_{2}+2\undertilde{d}{}_{4}\xi|^{2}+|\undertilde{d}{}_{3}-2\undertilde{d}{}_{4}\eta|^{2}\ \mathrm{d}x \\
	&\undertilde{=}
	|\undertilde{d}{}_{2}|^{2}+
	|\undertilde{d}{}_{3}|^{2}+
	|\undertilde{d}{}_{4}|^{2}
	=\undertilde{p}{}_{1}B_{1}^{T}B_{1}\undertilde{p}{}_{1}^{T}+
	\undertilde{p}{}_{2}B_{1}^{T}B_{1}\undertilde{p}{}_{2}^{T}
	=\undertilde{p}^{T}B\undertilde{p}.
	\end{align*}
	Here $\undertilde{p}\in \mathbb{R}^{8}$ and
	
	\begin{equation*}
	B_{1}=
	\begin{bmatrix}
	-\frac{\beta}{4} & -\frac{\beta-2}{4} & -\frac{\beta}{4} & \frac{\beta+2}{4} \\
	\frac{\alpha+2}{4} & -\frac{\alpha}{4} & \frac{\alpha-2}{4} & -\frac{\alpha}{4} \\
	-\frac{3}{8} & \frac{3}{8} & -\frac{3}{8} & \frac{3}{8} \\
	\end{bmatrix},
	\quad
	B=
	\begin{bmatrix}
	B_{1}^{T}B_{1} & \Large{0} \\
	\Large{0} & B_{1}^{T}B_{1} \\
	\end{bmatrix}.
	\end{equation*}
	
	Similarly, let $\sqcap_{K}\usigma=
	g_{1}\cdot\nabla\xi+
	g_{2}\cdot\nabla\eta+
	g_{3}\cdot\xi\nabla\eta+
	g_{4}\cdot\eta\nabla\xi$, recalling $\{g_i\}_{i=1:4}$ in Theorem \ref{thm:commutativity on a cell}, then there exists a semi-positive matrix $D$ with $\mathcal{O}(1)$ elements such that
	$
	|\sqcap_{K}\undertilde{\sigma}|_{1, K}^{2}=
	\undertilde{p}^{T}D\undertilde{p}
	$.
	
	Since  $\sqcap_{Q}$ is $H^{1}$ stable (see \cite{zhang2016stable}), it remains to show that
	$
	|\sqcap_{K}\undertilde{\sigma}|_{1, K}
	\leqslant
	C
	|\sqcap_{Q}\undertilde{\sigma}|_{1, K}
	$.
	
	We first show that $\ker B\subset\ \ker D$.
	Given a
	$\undertilde{p} \in \ker B$,
	then
	$|\sqcap_{Q}\usigma|_{1, K}^{2}=0$ and $\sqcap_{Q}\usigma$ is a constant vector.
	Since
	$\sqcap_{K}\usigma=\sqcap_{K}(\sqcap_{Q}\usigma)$,
	thus $\sqcap_{K}\usigma$ is also a constant vector and 	
	$|\sqcap_{K}\usigma|_{1, K}^{2}=0$, i.e. $\undertilde{p}\in \ker D$.
	
	Subsequently, we calculate all the eigenvalues of
	$B^{T}_{1}B_{1}$,
	(below we use $\lambda_{i}=\lambda_{i}(B_{1}^{T}B_{1})$, $i=1:4$)
	\begin{align*}
	\lambda_{1}&=0
	&
	\lambda_{3}&=
	\frac{17}{32}+\frac{\alpha^{2}+\beta^{2}}{8}-
	\frac
	{
		\sqrt{
			16(\alpha^{4}+\beta^{4})+
			32\alpha^{2}\beta^{2}+
			136(\alpha^{2}+\beta^{2})+1
		}
	}
	{32}  \\
	\lambda_{2}&=\frac{1}{2}
	&
	\lambda_{4}&=
	\frac{17}{32}+\frac{\alpha^{2}+\beta^{2}}{8}+
	\frac
	{
		\sqrt{
			16(\alpha^{4}+\beta^{4})+
			32\alpha^{2}\beta^{2}+
			136(\alpha^{2}+\beta^{2})+1
		}
	}
	{32}.
	\end{align*}
	The eigenvalues
	$\lambda(B)=
	\{\lambda_{1}, \lambda_{2}, \lambda_{3}, \lambda_{4}\}$
	of matrix $B$ are all double eigenvalues.
	Let
	$\{\undertilde{\nu}{}_{1}$,
	$\undertilde{\nu}{}_{2}\}$
	be two eigenvectors subordinating to the eigenvalue
	$\lambda_{1}=0$.
	Then we can decompose
	$
	\mathbb{R}^8=
	\mathrm{span}
	\{\undertilde{\nu}{}_{1}, \undertilde{\nu}{}_{2}\}
	\oplus
	(\mathrm{span}
	\{\undertilde{\nu}{}_{1}, \undertilde{\nu}{}_{2}\})^{\bot}
	$.
	Suppose
	$
	\undertilde{p}=
	\undertilde{\psi}{}_{1}+
	\undertilde{\psi}{}_{2}
	$,
	$\undertilde{\psi}{}_{1}
	\in
	\mathrm{span}\{\undertilde{\nu}{}_{1}
	$,
	$\undertilde{\nu}{}_{2}\}$,
	$\undertilde{\psi}{}_{2}
	\in
	(\mathrm{span}\{\undertilde{\nu}{}_{1}
	,\undertilde{\nu}{}_{2}\})^{\bot}$.
	Rayleigh quotient theorem reads
	\begin{align*}
	|\sqcap_{K}\usigma|_{1, K}^{2}
	&=
	\up^{T}D\up=\upsi{}_{2}^{T}D\upsi{}_{2}\leqslant
	\lambda_{max}(D)|\upsi{}_{2}|^{2} \\
	|\sqcap_{Q}\usigma|_{1, K}^{2}
	& \undertilde{=}
	\up^{T}B\up=\upsi{}_{2}^{T}B\upsi{}_{2}\geqslant
	\text{min}
	\{\lambda_{2}, \lambda_{3}, \lambda_{4}\}|\upsi{}_{2}|^{2}.
	\end{align*}
	This finishes the proof.
\end{proof}

\paragraph{\bf Proof of Theorem \ref{thm:L2 norm for interpolation estimation in H(rot)}}
Dividing a quadrilateral $K$ along the diameter into two triangles, we have the following estimation from the similar argument referred to Theorem $5.1$ in \cite{carstensen2012explicit}.
\begin{equation*}
\| \usigma-\sqcap_{K} \usigma\|_{0, K}
\leqslant Ch_{K}\|\nabla(\usigma-\sqcap_{K}\usigma)\|_{0, K}.
\end{equation*}
This proves \eqref{inequ:L2 norm for local estimation in H(rot)} by Lemma \ref{lemma:H^{1} stable on interpolation associated with QRT}.

\begin{theorem}
	Let $K$ be a convex quadrilateral,
	then it holds
	\begin{equation*}
	\| q-P_{K}q\|_{0, K}
	\leqslant
	Ch_{K}|q|_{1, K}
	\quad
	\forall
	q\in H^1(K).
	\end{equation*}
\end{theorem}

\begin{proof}
	The proof can be found in \cite{payne1960optimal}.
\end{proof}

\subsubsection{Interpolation error estimations in energy norms}
\label{subsubsec:interpolation error in energy norms}

\begin{theorem}
	Let $K$ be a convex quadrilateral, then it holds
	\begin{equation}
	|u-J_{K}u|_{1,K}
	\leqslant
	Ch_{K}|u|_{2, K}
	\quad
	\forall u\in H^2(K),
	\end{equation}
	and
	\begin{equation}
	|\usigma-\sqcap_{K}\usigma|_{\mathrm{rot},K}
	\leqslant Ch_{K}|\mathrm{rot}\usigma|_{1,K}
	\quad\forall\usigma\in H^1({\rm rot},K).
	\end{equation}
	\label{thm:inequ:energy norm for interpolation estimation in H1}
\end{theorem}

\begin{proof}
	We prove the estimations by the commutative diagrams.
	Since $\nabla (J_K)=\sqcap_{K}(\nabla)$ on $H^2(K)$, we have
	\begin{equation*}
	|u-J_Ku|_{1,K} =
	\|\nabla u -\sqcap_{K}(\nabla u)\|_{0,K}
	\leqslant
	Ch_K|u|_{2,K}.
	\end{equation*}
	Similarly, since ${\rm rot}\sqcap_{K}=P_K{\rm rot}$ on $\uH^1(K)$, we have
	\begin{equation*}
	\|\mathrm{rot} \usigma-\mathrm{rot}(\sqcap_{K}\usigma)\|_{0,K}=
	\|\mathrm{rot} \usigma-P_K(\mathrm{rot}\usigma)\|_{0,K}\leqslant
	Ch_K|\mathrm{rot} \usigma|_{1,K}.
	\end{equation*}
	The proof is completed.
\end{proof}

\subsection{Finite element spaces on a grid $\mathcal{J}_h$}
\label{subsec:FE sapces on a grid}

\begin{definition}
	Associated with the
	${\rm QBL}$ element,
	define the finite element spaces
	$V_{h}^{\rm QBL}$ and $V_{h0}^{\rm QBL}$ by
	\begin{equation*}
	V_{h}^{\rm QBL}
	\triangleq
	\{
	v_{h}\in L^{2}(\Omega):
	v_{h}|_{K}\in P_{K}^{\rm QBL},
	v_{h}\ \text{is continuous at two endpoints of edge}\ e\in\mathcal{E}^{i}_{h}
	\},
	\end{equation*}
	\begin{equation*}
	\mbox{and}\ \ \
	V_{h0}^{\rm QBL}
	\triangleq
	\{
	v_{h}\in V_{h}^{\rm QBL}:
	v_{h}=0\ \text{at two endpoints of edge}\ e \in\mathcal{E}^{b}_{h}
	\}.
	\end{equation*}
\end{definition}

\begin{definition}
	Associated with the
	${\rm QRT}$ element,
	define the finite element spaces
	$V_{h}^{\rm QRT}$ and $V_{h0}^{\rm QRT}$ by
	\begin{equation*}
	V_{h}^{\rm QRT}
	\triangleq
	\{
	\utau{}_{h}
	\in \uL^2(\Omega):
	\utau{}_{h}|_{K}
	\in P_{K}^{\rm QRT},
	\dashint_{e}
	\utau{}_{h}
	\cdot
	\ut{}_{e}
	\ \mathrm{d}s
	\ \mbox{is continuous at the edge}\ e\in\mathcal{E}^{i}_{h}
	\},
	\end{equation*}
	\label{def:P_h^rot,Omega}
	\begin{equation*}
	\mbox{and}\ \ \ V_{h0}^{\rm QRT}
	\triangleq
	\{
	\undertilde{\tau}{}_{h}
	\in V_{h}^{\rm QRT}:
	\dashint_{e}
	\undertilde{\tau}{}_{h}\cdot\undertilde{t}{}_{e}
	\ \mathrm{d}s=0
	\ \text{at the edge}\ e \in\mathcal{E}^{b}_{h}
	\}.
	\end{equation*}
\end{definition}

\begin{definition}
	Define the piecewise constant finite element spaces
	$W_h$ and $W_{h0}$ by
	\begin{equation*}
	W_h\triangleq
	\{q_h\in L^2(\Omega):q_h|_K=P_Kq, q\in L^2(\Omega)\},
	\quad \mbox{and}\ \	W_{h0}\triangleq
	\{q_h\in W_h:\int_{\Omega}q_h\ \mathrm{d}x = 0\}.
	\end{equation*}
\end{definition}


The properties on a single cell can be generalized on a grid. We firstly adopt a lemma below.
\begin{lemma}
	It holds on the subdivision that
	\begin{equation}\label{eq:infsup1}
	\sup_{\utau{}_h\in V^{\rm QRT}_h}(\rot_h\utau{}_h,q_h)\geqslant C\|\utau{}_h\|_{\rot,h}\|q_h\|_{0,\Omega},\ \ \mbox{for\ any}\ q_h\in W_h,
	\end{equation}
	and
	\begin{equation}\label{eq:infsup2}
	\sup_{\utau{}_h\in V^{\rm QRT}_{h0}}(\rot_h\utau{}_h,q_h)\geqslant C\|\utau{}_h\|_{\rot,h}\|q_h\|_{0,\Omega},\ \ \mbox{for\ any}\ q_h\in W_{h0}.
	\end{equation}
\end{lemma}
\begin{proof}
	Given $q_h\in W_h$, there exists a $\utau\in \uH{}^1(\Omega)$, such that $\rot \utau=q_h$, and $\|\utau\|_{1,\Omega}\leqslant C\|q_h\|_{0,\Omega}$. Set $\utau{}_h:=\sqcap_h\utau$, then $\rot_h\utau{}_h=q_h$, and $\|\utau{}_h\|_{\rot,h}\leqslant C\|\utau\|_{1,\Omega}$. This proves \eqref{eq:infsup1}. Similarly is \eqref{eq:infsup2} proved.
\end{proof}

\begin{theorem}
	The commutative diagrams hold as below:
	\begin{equation}
	\begin{array}{ccccccccc}
	\mathbb{R} & ~~~\longrightarrow~~~ & H^2(\Omega) & ~~~\xrightarrow{\bs{\mrm{\nabla}}}~~~ & \uH^1(\Omega) & ~~~\xrightarrow{\mrm{rot}}~~~ & L^2(\Omega)  \\
	& & \downarrow J_h & & \downarrow \sqcap_{h} & & \downarrow P_h & & \\
	\mathbb{R} & \longrightarrow & V^{\rm QBL}_h & \xrightarrow{\bs{\mrm{\nabla_h}}} & V^{\rm QRT}_h & \xrightarrow{\mrm{rot_h}} & W_h,
	\end{array}
	\label{eq:ddrc1}
	\end{equation}
	and
	\begin{equation}
	\begin{array}{ccccccccc}
	\{0\} & ~~~\longrightarrow~~~ & H^2(\Omega)\cap H^1_0(\Omega) & ~~~\xrightarrow{\bs{\mrm{\nabla}}}~~~ & \uH^1(\Omega)\cap H_0(\rot,\Omega) & ~~~\xrightarrow{\mrm{rot}}~~~ & L_0^2(\Omega)  & ~~~\xrightarrow{\int_\Omega\cdot}~~~ & \{0\}  \\
	& & \downarrow J_h & & \downarrow \sqcap_{h} & & \downarrow P_h & & \\
	\{0\} & \longrightarrow & V^{\rm QBL}_{h0} & \xrightarrow{\bs{\mrm{\nabla_h}}} & V^{\rm QRT}_{h0} & \xrightarrow{\mrm{rot_h}} & W_{h0} & \xrightarrow{\int_\Omega\cdot} & \{0\}.
	\end{array}
	\label{eq:ddrc2}
	\end{equation}
\end{theorem}
%
\begin{proof}
	We first consider \eqref{eq:ddrc1}. The commutativity is trivial by Theorem \ref{thm:commutativity on a cell} and it remains us to verify the discretized de Rham complex by the standard dimension counting technique.
	
	Evidently ${\rm ker}(\nabla_h)=\mathbb{R}$ and $\nabla_h V^{\rm QBL}_h\subset V_{h}^{\rm QRT}$. On the other hand, by \eqref{eq:infsup1}, $\rot_h V_{h}^{\rm QRT}=W_h$. This way, \eqref{eq:ddrc1} follows by noting that $\dim(V^{\rm QBL}_h)=\#(\mathcal{N}_h)$ and $\dim(V_{h}^{\rm QRT})=\#(\mathcal{E}_h)$, and that $\dim(V_{h}^{\rm QRT})=\dim(V^{\rm QBL}_h)+\dim(W_h)-1$ by the Euler formula.
	
	Similarly is \eqref{eq:ddrc2} proved. The proof is completed.
\end{proof}

The error estimation of the global interpolator is the same as that of the respective local ones.
\begin{theorem}
	There exists a constant $C$ depending on the shape regularity of $\mathcal{J}_h$ only, such that
	\begin{enumerate}
		\item $\displaystyle\|u-J_{h}u\|_{0,\Omega}+h|u-J_hu|_{1,h}
		\leqslant
		Ch^2|u|_{2,\Omega}
		\quad\forall u\in H^2(\Omega);$
		\item $\displaystyle	\|\usigma-\sqcap_{h}\usigma\|_{{\rm rot},h}
		\leqslant
		Ch(|\usigma|_{1,\Omega}+|{\rm rot}\usigma|_{1,\Omega})
		\quad\forall \usigma\in H^1({\rm rot},\Omega);$
		\item $	\displaystyle\| q-P_{h}q\|_{0,\Omega}
		\leqslant
		Ch|q|_{1,\Omega}
		\quad\forall q\in H^1(\Omega).$
	\end{enumerate}
	\label{thm:global intepolation error estimation}
\end{theorem}
\section{Nonconforming finite element spaces and their modulus of continuity} 
\label{sec:NFEs and mc}
In this section, we show that on a grid that consists of arbitrary quadrilaterals and satisfies the condition that the cells are asymptotically parallelograms, the spaces $V^{\rm QBL}_h$ and $V^{\rm QRT}_h$, though not subspaces of  $H^1$ and $H(\rot)$ respectively, the consistency can be controlled well. We begin with an analysis that $V^{\rm QBL}_h$ is in general not continuous.
\subsection{Continuity and non-continuity of $V_{h0}^{\rm QBL}$}
\label{subsec:cnc of QBL}

Let $\mathcal{G}_D$ be a patch with the center $D$ and four cells $K_{1},K_{2},K_{3},K_{4}$, see Figure \ref{fig:patch for QBL FE space}. Let $V_{h0}^{\rm QBL}(\mathcal{G}_D)$ be ${\rm QBL}$ finite element space defined on $\mathcal{G}_D$ with zero boundary condition. Denote by local shape parameters $\alpha_i,\beta_i$ of $K_i$, $i=1:4$. Here $\alpha_i=\beta_i=0$ for $i=2:4$.

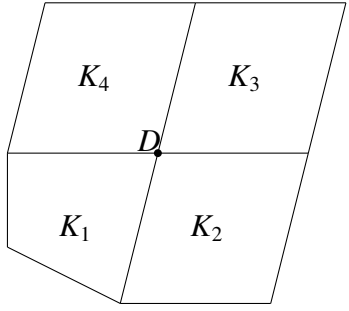
\begin{figure}[htbp]
	\centering
	\begin{tikzpicture}
	\draw(-1.5,-1.25)--(0,-2);
	\draw(0,-2)--(2,-2);
	\draw(-1.5,0)--(0.5,0);
	\draw(0.5, 0)--(2.5, 0);
	\draw(-1, 2)--(1,2);
	\draw(1,2)--(3, 2);
	
	\draw(-1.5,-1.25)--(-1.5,-0);
	\draw(-1.5,-0)--(-1, 2);
	\draw(0,-2)--(0.5,0);
	\draw(0.5, 0)--(1, 2);
	\draw(2,-2)--(2.5,-0);
	\draw(2.5, -0)--(3, 2);

	\draw(0.38,0.17)node{$D$};
	\draw(-0.6,-1)node{$K_{1}$};
	\draw(1.15,-1)node{$K_{2}$};
	\draw(-0.35,1)node{$K_{4}$};
	\draw(1.65,1)node{$K_{3}$};
	\fill(0.5,0) circle(1.5pt);
	\end{tikzpicture}
	\caption{Illustration of a patch $\mathcal{G}_D$}
	\label{fig:patch for QBL FE space}
\end{figure}

\begin{theorem}
	For $\alpha_1,\beta_1\neq 0$ and $v_h\neq 0\in V_{h0}^{\rm QBL}(\mathcal{G}_D)$, there exists a function $\varphi\in C^{\infty}(\mathcal{G}_D)$ such that
	\begin{equation*}
	(\nabla\varphi,\nabla_{h} v_h)+(\Delta \varphi,v_h)\neq 0
	\end{equation*}
\end{theorem}
\begin{proof}
	Since $\alpha_i,\beta_i=0$ for $i=2:4$, then
	\begin{equation*}
	(\nabla \varphi,\nabla_h v_h)+(\Delta \varphi,v_h)=
	\int_{\partial K_1}\frac{\partial \varphi}{\partial n}(v_h-q)\ \mathrm{d}s.
	\end{equation*}
	Here $q$ is the linear interpolation of $v_h$ on $e$ with respect to  endpoints. Without loss of generality, we assume $\alpha_1\neq 0$. Noticing that $\dim V_{h0}^{\rm QBL}(\mathcal{G}_D)=1$, we denote by $\phi_{D}$ the  basis of $V_{h0}^{\rm QBL}(\mathcal{G}_D)$ and $\phi_{D}|_{K_1}=\phi_{1}$ (see \eqref{equ:local dual basis of QBL}), then $v_h|_{K_1}=v_h(D)\phi_{1}$.
	Let $\varphi$ be a linear polynomial whose gradient is $\ur{}_{K_1}/\ur{}_{K_1}\times\us{}_{K_1}$, defined on the patch $\mathcal{G}_D$. Then the proof is completed by simple calculation.
\end{proof}

\subsection{Modulus of continuity of $V_{h0}^{\rm QBL}$}
\label{subsec:mc of QBL}

Define consistency functional $E(\uzeta,v_{h})$ by
\begin{align}
E(\uzeta, v_{h})&=(\uzeta,\nabla_{h}v_{h})+(\dv\uzeta,v_{h})
\quad
\mbox{for}\
\uzeta\in \uH^{1}(\Omega),
v_{h}\in V_{h0}^{\rm QBL} \\
E(\uzeta, v_{h})&=(\uzeta,\nabla_{h}v_{h})+(\dv\uzeta,v_{h})
\quad
\mbox{for}\
\uzeta\in \uH{}_{0}^{1}(\Omega),
v_{h}\in V_{h}^{\rm QBL} \cap L^2_0(\Omega).
\end{align}

Let $\{\phi_{i}\}_{i=1:4}$ be local basis of $P_{K}^{\rm QBL}$, then define by
\begin{equation*}
\iota_{1}=\sum_{i=1}^{4}\phi_{i}^{(1)}v_{h}(A_{i})\quad
\iota_{2}=\sum_{i=1}^{4}\phi_{i}^{(2)}v_{h}(A_{i})\quad
\iota_{3}=\sum_{i=1}^{4}\phi_{i}^{(3)}v_{h}(A_{i}).
\end{equation*}

\begin{theorem}
	For
	$\uzeta\in \uH^{1}(\Omega)$,
	$v_{h}\in V_{h0}^{\rm QBL}$ or
	$\uzeta\in \uH{}_{0}^{1}(\Omega)$,
	$v_{h}\in V_{h}^{\rm QBL}$,
	it holds
	\begin{equation*}
	E(\uzeta,v_{h})
	\leqslant
	Ch\|\uzeta\|_{1,\Omega}|v_{h}|_{1,h}.
	\end{equation*}
	\label{thm:global estimate of E(u,v_h)}
\end{theorem}
\begin{proof}
	Evidently the consistency functional can be decomposed into
	\begin{equation*}
	E(\uzeta, v_{h})=
	\sum_{e\in\mathcal{E}_{h}}
	\int_{e}\uzeta\cdot\un{}_{e}[\![v_{h}-q]\!]_{e}\ \mathrm{d}s=
	\sum_{K\in\mathcal{J}_{h}}\sum_{e\subset\partial K}
	\int_{e}\uzeta\cdot\un{}_{e}(v_{h}-q)\ \mathrm{d}s.
	\end{equation*}
	Here $q$ is the linear interpolation of $v_h$ on $e$ with respect to  endpoints. Then by direct calculation, we have
	\begin{align*}
	\int_{\partial K}\uzeta\cdot\un(v_{h}-q)\ \mathrm{d}s
	&=
	\beta_K\iota_{1}\int_{e_{1}}\uzeta\cdot\un{}_{1}
	(\frac{\xi^2}{1+\alpha_K}-(1+\alpha_K))\ \mathrm{d}s+
	\alpha_K\iota_{1}\int_{e_{2}}\uzeta\cdot\un{}_{2}
	(\frac{\eta^2}{-1+\beta_K}-(-1+\beta_K))\ \mathrm{d}s \\
	&+
	\beta_K\iota_{1}\int_{e_{3}}\uzeta\cdot\un{}_{3}
	(\frac{\xi^2}{-1+\alpha_K}-(-1+\alpha_K))\ \mathrm{d}s+
	\alpha_K\iota_{1}\int_{e_{4}}\uzeta\cdot\un{}_{4}
	(\frac{\eta^2}{1+\beta_K}-(1+\beta_K))\ \mathrm{d}s \\
	&\leqslant
	Ch\iota_{1}\|\uzeta\|_{1,K}
	\label{equ:consistency functional}.
	\end{align*}	
	On the other hand, setting
	$
	\undertilde{\iota}=
	(\iota_{1},\iota_{2},\iota_{3})^{T}
	$ and
	$|v_{h}|^{2}_{1, K}=\undertilde{\iota}^{T}G\undertilde{\iota}$,
	then
	\begin{equation*}
	G=
	\begin{bmatrix}
	\frac
	{4(1+\beta_K^{2}) \undertilde{s}{}_K\cdot\undertilde{s}{}_{K}
		-8\alpha_K\beta_K   \undertilde{r}{}_{K}\cdot\undertilde{s}{}_{K}
		+4(1+\alpha_K^{2})\undertilde{r}{}_{K}\cdot\undertilde{r}{}_{K}}
	{3              \undertilde{r}{}_{K}\times\undertilde{s}{}_{K}} &
	\frac
	{4\alpha_K \undertilde{s}{}_{K}\cdot\undertilde{s}{}_{K}
		-4\beta_K  \undertilde{r}{}_{K}\cdot\undertilde{s}{}_{K}}
	{3       \undertilde{r}{}_{K}\times\undertilde{s}{}_{K}} &
	\frac
	{-4\alpha_K \undertilde{r}{}_{K}\cdot\undertilde{s}{}_{K}
		+4\beta_K  \undertilde{r}{}_{K}\cdot\undertilde{r}{}_{K}}
	{3        \undertilde{r}{}_{K}\times\undertilde{s}{}_{K}}  \\
	\frac
	{4\alpha_K \undertilde{s}{}_{K}\cdot\undertilde{s}{}_{K}
		-4\beta_K  \undertilde{r}{}_{K}\cdot\undertilde{s}{}_{K}}
	{3       \undertilde{r}{}_{K}\times\undertilde{s}{}_{K}} &
	\frac
	{4\undertilde{s}{}_{K}\cdot\undertilde{s}{}_{K}}
	{ \undertilde{r}{}_{K}\times\undertilde{s}{}_{K}} &
	\frac
	{-4\undertilde{r}{}_{K}\cdot\undertilde{s}{}_{K}}
	{  \undertilde{r}{}_{K}\times\undertilde{s}{}_{K}} \\
	\frac
	{-4\alpha_K \undertilde{r}{}_{K}\cdot\undertilde{s}{}_{K}
		+4\beta_K  \undertilde{r}{}_{K}\cdot\undertilde{r}{}_{K}}
	{3        \undertilde{r}{}_{K}\times\undertilde{s}{}_{K}} &
	\frac
	{-4\undertilde{r}{}_{K}\cdot\undertilde{s}{}_{K}}
	{  \undertilde{r}{}_{K}\times\undertilde{s}{}_{K}} &
	\frac
	{4 \undertilde{r}{}_{K}\cdot\undertilde{r}{}_{K}}
	{  \undertilde{r}{}_{K}\times\undertilde{s}{}_{K}} \\
	\end{bmatrix}.
	\end{equation*}
	By the generalized Rayleigh quotient theorem, then for
	$
	F=
	\begin{bmatrix}
	1 & 0 & 0 \\
	0 & 0 & 0 \\
	0 & 0 & 0 \\
	\end{bmatrix}
	$ and all $\undertilde{\iota}\in R^3$
	\begin{equation*}
	\undertilde{\iota}^{T}F\undertilde{\iota}
	\leqslant
	\lambda_{F,G}\undertilde{\iota}^{T}G\undertilde{\iota},
	\quad
	\lambda_{F,G}=
	\frac
	{9\ur{}_K\times\us{}_K}
	{4((3-\alpha_K^2+3\beta_K^2)\us{}_{k}\cdot\us{}_{k}
		-4\alpha_K\beta_K\ur{}_{K}\cdot\us{}_{K}
		+(3+3\alpha_K^2-\beta_K^2)\ur{}_{K}\cdot\ur{}_{K})}.
	\end{equation*}
	
	In summary, we have
	\begin{equation*}
	E(\uzeta,v_h)\leqslant
	Ch\sum_{K\in\mathcal{J}{}_{h}}\|\uzeta\|_{1,K}|v_h|_{1,K}
	\leqslant Ch\|\uzeta\|_{1,\Omega}|v_h|_{1,h}.
	\end{equation*}
	The proof is completed.
\end{proof}

\subsection{Modulus of continuity of $V_{h0}^{\rm QRT}$}
\label{subsec:mc of QRT}

Define consistency functional
$E(w,\undertilde{\tau}{}_{h})$ by
\begin{align}
E(w,\undertilde{\tau}{}_{h})&=
(w,\mathrm{rot}_{h}\undertilde{\tau}{}_{h})-
(\undertilde{\mathrm{curl}}\ w,\undertilde{\tau}{}_{h})
\quad
\mbox{for}\
w\in H^{1}(\Omega),
\forall
\undertilde{\tau}{}_{h}\in V_{h0}^{\rm QRT} \\
E(w,\undertilde{\tau}{}_{h})&=
(w,\mathrm{rot}_{h}\undertilde{\tau}{}_{h})-
(\undertilde{\mathrm{curl}}\ w,\undertilde{\tau}{}_{h})
\quad
\mbox{for}\
w\in H^{1}_0(\Omega),
\forall
\undertilde{\tau}{}_{h}\in V_{h}^{\rm QRT}.
\end{align}
Evidently the consistency functional can be decomposed into
\begin{equation}
E(w,\undertilde{\tau}{}_{h})=
\sum_{K\in\mathcal{J}{}_{h}}\sum_{e\subset\partial K}
\int_{e}
(w-c_K)
(\undertilde{\tau}{}_{h}
\cdot
\ut{}_{e}-
P_{e}
(\utau{}_{h}
\cdot
\ut{}_{e}))
\ \mathrm{d}s.
\label{equ:consistency functional decomposition}
\end{equation}
Here $c_K$ is an arbitrary constant.

\begin{theorem}
	For
	$w\in H^{1}(\Omega)$,
	$\utau{}_{h}
	\in V_{h0}^{\rm QRT}$ or
	$w\in H_0^{1}(\Omega)$,
	$\utau{}_{h}
	\in V_{h}^{\rm QRT}$,
	it holds
	\begin{equation}
	E(w,\utau{}_{h})\leqslant Ch
	|w|_{1,\Omega}\|\utau{}_{h}\|_{{\rm rot},h}.
	\label{inequ: rot consistency error}
	\end{equation}
	\label{thm: rot consistency error}
\end{theorem}
\begin{proof}
	Evidently, we have
	$\inf_{c_{K}\in \mathbb{R}}\|w-c_K\|_{0,\partial K}^{2}\leqslant Ch_{K}|w|_{1,K}^{2}$.
	Let
	$
	\utau{}_{h}=
	\gamma_{1}\nabla\xi+
	\gamma_{2}\nabla\eta+
	\gamma_{3}\hat{\xi}\nabla\eta+
	\gamma_{4}\hat{\eta}\nabla\xi
	$ and $\|\utau{}_{h}\|_{{\rm rot},K}^2=\ugamma^T V\ugamma$, with  $\ugamma=(\gamma_{1},\gamma_{2},\gamma_{3},\gamma_{4})^T$ and	
	\begin{equation}
	V=
	\begin{bmatrix}
	\frac
	{4 \us{}_K\cdot\us{}_K}
	{  \ur{}_K\times\us{}_K} &
	\frac
	{-4\ur{}_K\cdot\us{}_K}
	{  \ur{}_K\times\us{}_K} &
	0 &
	0                        \\
	\frac
	{-4\ur{}_K\cdot\us{}_K}
	{  \ur{}_K\times\us{}_K} &
	\frac
	{4 \ur{}_K\cdot\ur{}_K}
	{  \ur{}_K\times\us{}_K} &
	0 &
	0                         \\
	0 &
	0 &
	\frac
	{4(3+3\alpha_K^2-\beta_K^2)\ur{}_K\cdot\ur{}_{K}+36}
	{9                   \ur{}_K\times\us{}_K} &
	\frac
	{-8\alpha_K\beta_K \ur{}_K\cdot\us{}_K-36}
	{9             \ur{}_K\times\us{}_K}   \\
	0 &
	0 &
	\frac
	{-8\alpha_K\beta_K \ur{}_K\cdot\us{}_K-36}
	{9             \ur{}_K\times\us{}_K} &
	\frac
	{4(3+3\beta_K^2-\alpha_K^2)\us{}_K\cdot\us{}_K+36}
	{9                   \ur{}_K\times\us{}_K}
	\end{bmatrix}.
	\label{matrix:V}
	\end{equation}
	
	On the other hand
	\begin{equation*}
	\|
	\utau{}_{h}
	\cdot
	\ut{}_{e}-
	P_{e}
	(\utau{}_{h}
	\cdot
	\ut{}_{e})
	\|_{0,e}^{2}
	=
	\int_{e}
	(\utau{}_{h}
	\cdot
	\ut{}_{e}-
	P_{e}
	(\utau{}_{h}
	\cdot
	\ut{}_{e}))^2
	\ \mathrm{d}s
	=
	\int_{e}
	\utau{}_{h}
	\cdot
	\ut{}_{e}
	(\utau{}_{h}
	\cdot
	\ut{}_{e}-
	P_{e}
	(\utau{}_{h}
	\cdot
	\ut{}_{e}))
	\ \mathrm{d}s,
	\end{equation*}
	then by simple calculation
	\begin{align*}
	&\int_{e_{1}}
	\undertilde{\tau}{}_{h}
	\cdot
	\undertilde{t}{}_{1}
	(\undertilde{\tau}{}_{h}
	\cdot
	\undertilde{t}{}_{1}-
	P_{e_{1}}
	(\undertilde{\tau}{}_{h}
	\cdot
	\undertilde{t}{}_{1}))
	\ \mathrm{d}s
	\\
	=
	&\int_{e_{1}}
	(
	\gamma_{3}\nabla\eta\cdot
	\undertilde{t}{}_{1}\xi+
	\gamma_{4}\nabla\xi\cdot
	\undertilde{t}{}_{1}\eta
	)
	\cdot
	(
	\gamma_{3}\nabla\eta\cdot
	\undertilde{t}{}_{1}
	(\xi-P_{e_{1}}(\xi))+
	\gamma_{4}\nabla\xi \cdot
	\undertilde{t}{}_{1}
	(\eta-P_{e_{1}}(\eta))
	)
	\ \mathrm{d}s \\
	=
	&\int_{e_{1}}
	(\gamma_{3}\nabla\eta\cdot
	\undertilde{t}{}_{1})^{2}\xi^{2}+
	2(\gamma_{3}\nabla\eta\cdot
	\undertilde{t}{}_{1})
	(\gamma_{4}\nabla\xi\cdot
	\undertilde{t}{}_{1})\xi\eta+
	(\gamma_{4}\nabla\xi\cdot
	\undertilde{t}{}_{1})^{2}\eta(\eta-1)
	\ \mathrm{d}s =
	\frac{4(1+\alpha_K)^{2}\beta_K^{2}}{3|e_{1}|}
	\undertilde{\gamma}^{T}U\undertilde{\gamma}.
	\end{align*}
	Here $U=
	\begin{bmatrix}
	0 & 0 & 0 & 0 \\
	0 & 0 & 0 & 0 \\
	0 & 0 & 1 & 1 \\
	0 & 0 & 1 & 1 \\
	\end{bmatrix}$. Similarly
	\begin{align*}
	&
	\|\utau{}_{h}\cdot\ut{}_{2}-
	P_{e_{2}}({\utau}{}_{h}\cdot\ut{}_{2})\|_{0,e_{2}}^{2}
	=
	\frac{4\alpha_K^{2}(-1+\beta_K)^{2}}{3|e_{2}|}
	\ugamma^{T}U\ugamma	\\
	&\|
	\utau{}_{h}\cdot\ut{}_{3}-
	P_{e_{3}}(\utau{}_{h}\cdot\ut{}_{3})\|_{0,e_{3}}^{2}
	=
	\frac{4(-1+\alpha_K)^{2}\beta_K^{2}}{3|e_{3}|}
	\ugamma^{T}U\ugamma	\\
	&\|
	\utau{}_{h}\cdot\ut{}_{4}-
	P_{e_{4}}(\utau{}_{h}\cdot\ut{}_{4})\|_{0,e_{4}}^{2}
	=
	\frac{4\alpha_K^{2}(1+\beta_K)^{2}}{3|e_{4}|}
	\ugamma^{T}U\ugamma.
	\end{align*}
	By the generalized Rayleigh quotient theorem, then for all
	$\ugamma \in \mathbb{R}^4$
	\begin{equation*}
	\|
	\utau{}_{h}\cdot\ut{}_{e}-
	P_{e}(\utau{}_{h}\cdot\ut{}_{e})\|_{0,\partial {K}}^{2}
	\leqslant
	C\lambda_{U,V}h_K\|\utau{}_h\|_{{\rm rot},K}^2.
	\end{equation*}
	Here $\lambda_{U,V}=p(\alpha_K,\beta_K,\ur{}_K,\us{}_K)/q(\alpha_K,\beta_K,\ur{}_K,\us{}_K)$ with
	$$
	p(\alpha_K,\beta_K,\ur{}_K,\us{}_K) =
	9((3+3\alpha_K^{2}-\beta_K^2)\ur{}_K\cdot\ur{}_K+
	4\alpha_K\beta_K \ur{}_K\cdot\us{}_K+
	(3+3\beta_K^{2}-\alpha_K^2)\us{}_K\cdot\us{}_K+
	36)\ur{}_K\times\us{}_K
	$$
	and
	\begin{multline*}
	q(\alpha_K,\beta_K,\ur{}_K,\us{}_K) =
	4((3+3\alpha_K^{2}-\beta_K^2)(3+3\beta_K^{2}-\alpha_K^2)
	(\ur{}_K\cdot\ur{}_K)(\us{}_K\cdot\us{}_K)
	-4\alpha_K^{2}\beta_K^{2}(\ur{}_K\cdot\us{}_K)^2
	\\+(27-9\alpha_K^2+27\beta_K^2)\us{}_K\cdot\us{}_K-36\alpha_K\beta_K\ur{}_K\cdot\us{}_K
	+(27+27\alpha_K^2-9\beta_K^2)\ur{}_K\cdot\ur{}_K).
	\end{multline*}
	\eqref{inequ: rot consistency error} follows from the Cauchy-Schwarz inequality, the proof is completed.
\end{proof}

%
%

%
%
%

\section{Finite element schemes for respective model problems}
\label{sec:FE schemes for model problems}

\subsection{A finite element scheme for the Poisson equation}
\label{subsec:a FE scheme for the Poisson equation}

We consider the Poisson problem with homogeneous boundary condition
\begin{equation*}
\left\{
\begin{aligned}
-\Delta u & =f &  \quad & \mbox{in} \ \Omega \\
u & =0 &  \quad & \mbox{on} \ \Gamma 
\end{aligned}
\right.
\label{equ:Poisson problem in numerical experiment}
\end{equation*}

The variational formulation is to find $u\in H^1_0(\Omega)$, such that 
\begin{equation}
\int_\Omega \nabla u\nabla v \ \mathrm{d}x=\int_\Omega fv \ \mathrm{d}x,\quad\forall\,v\in H^1_0(\Omega).
\label{equ:variational problem on Poisson }
\end{equation}
The finite element problem is to find $u_h\in V_{h0}^{\rm QBL}$, such that 
\begin{equation}
\sum_{K\in\mathcal{J}_h}\int_K\nabla u_h\nabla v_h\ \mathrm{d}x=\int_\Omega fv_h \ \mathrm{d}x,\quad\forall\,v_h\in V_{h0}^{\rm QBL}.
\label{equ:discrete problem on Poisson}
\end{equation}
\begin{theorem}
	Let $u\in H^2(\Omega)\cap H^1_0(\Omega)$ and $u_h$ be the solutions of \eqref{equ:variational problem on Poisson }, and \eqref{equ:discrete problem on Poisson}, respectively. Then
	\begin{equation}
	|u-u_h|_{1,h}\leqslant Ch\|u\|_{2,\Omega}\ \ \ \mbox{and}\ \ \ \|u-u_h\|_{0,\Omega}\leqslant Ch^2\|u\|_{2,\Omega}.
	\end{equation}
\end{theorem}
\begin{proof}
	The former is evident by Theorem \ref{thm:global intepolation error estimation} and Theorem \ref{thm:global estimate of E(u,v_h)}. Then we are going to show the latter. Let $z\in H^2(\Omega)\cap H^1_0(\Omega)$ and $z_h\in V_{h0}^{\rm QBL}$ be the solutions of the two problems below, respectively,
	\begin{equation*}
		(\nabla v,,\nabla z)=(u-u_h,v),\ \forall v\in H_0^1(\Omega) 
		\ \ \ \mbox{and}\ \ \ 
		(\nabla_h v_h,\nabla_h z_h)=(u-u_h,v_h),\ \forall v_h\in V_{h0}^{\rm QBL}.
	\end{equation*}
	Then it holds that
	\begin{equation}
		\|u-u_h\|^2_{0,\Omega}=(\nabla_h (u-u_h),\nabla_h (z-z_h))+(\nabla_h (u-u_h),\nabla_h z_h)+(\nabla_h u_h,\nabla_h (z-z_h)).
		\label{equ:Nitsche}
	\end{equation}
	For the first term in the right side of \eqref{equ:Nitsche}, we utilize the regularity of solution on a convex domain, then it holds that
	\begin{equation}
		(\nabla_h (u-u_h),\nabla_h (z-z_h))\leqslant Ch^2\|u\|_{2,\Omega}\|z\|_{2,\Omega}\leqslant
		Ch^2\|u\|_{2,\Omega}\|u-u_h\|_{0,\Omega}.
		\label{inequ:first term in dual argument}
	\end{equation}
	For the second term in the right side of \eqref{equ:Nitsche}, repeating the similar arguments in Theorem \ref{thm:global estimate of E(u,v_h)} yields 
	\begin{equation*}
		(\nabla_h (u-u_h),\nabla_h z_h)\leqslant
		C\sum_{K\in\mathcal{J}{}_{h}}h_K^2\|u\|_{2,K}|z_h|_{2,K}\leqslant
		Ch\|u\|_{2,\Omega}(\sum_{K\in\mathcal{J}{}_{h}}h_K^2|z_h|^2_{2,K})^{\frac{1}{2}}.
	\end{equation*}
	Since $|z-J_{h}z|_{1,K}\leqslant Ch_{K}|z|_{2,K}$, and $|J_hz|_{2,K}\leqslant C|z|_{2,K}$ (by Theorem \ref{thm:commutativity on a cell} and Lemma \ref{lemma:H^{1} stable on interpolation associated with QRT}), then it holds that
	\begin{align*}
		\sum_{K\in\mathcal{J}{}_{h}}h_K^2|z_h|^2_{2,K}
		&\leqslant C\sum_{K\in\mathcal{J}{}_{h}}h_K^2(|z_h-J_hz|^2_{2,K}+|J_hz|^2_{2,K})\\
		&\leqslant
		C\sum_{K\in\mathcal{J}{}_{h}}(|z_h-J_hz|^2_{1,K}+h_K^2|J_hz|^2_{2,K})\\
		&\leqslant
		C\sum_{K\in\mathcal{J}{}_{h}}(|z-z_h|^2_{1,K}+|z-J_hz|^2_{1,K}+h_K^2|J_hz|^2_{2,K})\leqslant Ch^2\|z\|^2_{2,\Omega}.
	\end{align*}
	This way, we have $(\nabla_h (u-u_h),\nabla_h z_h)\leqslant Ch^2\|u\|_{2,\Omega}\|z\|_{2,\Omega}\leqslant Ch^2\|u\|_{2,\Omega}\|u-u_h\|_{0,\Omega}$. 
	Similarly, so does the third term in the right side of \eqref{equ:Nitsche} hold. Thus we have 
	\begin{equation}
		(\nabla_h (u-u_h),\nabla_h z_h)+(\nabla_h u_h,\nabla_h (z-z_h))\leqslant Ch^2\|u\|_{2,\Omega}\|u-u_h\|_{0,\Omega}.
		\label{inequ:second and third term in dual argument}
	\end{equation}
	Substituting \eqref{inequ:first term in dual argument} and \eqref{inequ:second and third term in dual argument} into \eqref{equ:Nitsche}, the proof is completed.
\end{proof}
\subsubsection{Numerical verification}

We choose the computational domain to be the quadrilateral with vertexes
$(0,0)$, $(1,0)$, $(2,2)$, $(-1,1)$. 
The data $f$ is chosen such that the exact solution is the polynomial 
$u = y(x+y)(x-3y+4)(2x-y-2)$. 
We subdivide the domain with quadrilateral grids and triangle grids,  respectively, 
and numerical solutions are computed on both grids. 
To generate quadrilateral grids, 
we use bisection strategy.
To generate triangle grids, 
we firstly subdivide the domain with quadrilateral grids, 
then bisect all of them each to two triangles, see Figure 
\ref{fig:Poissonshowmesh}. We first test the performance of the ${\rm QBL}$ element on the quadrilateral grids, then we test Courant element on the triangle grids as a comparison. The results are recorded in Table \ref{tab:tab6}.

\begin{figure}[htbp]
	\vspace{-2cm}
	\centering
	\subfigure[]{
		\begin{minipage}[t]{0.23\linewidth}
			\centering
			\includegraphics[width=1.651in]{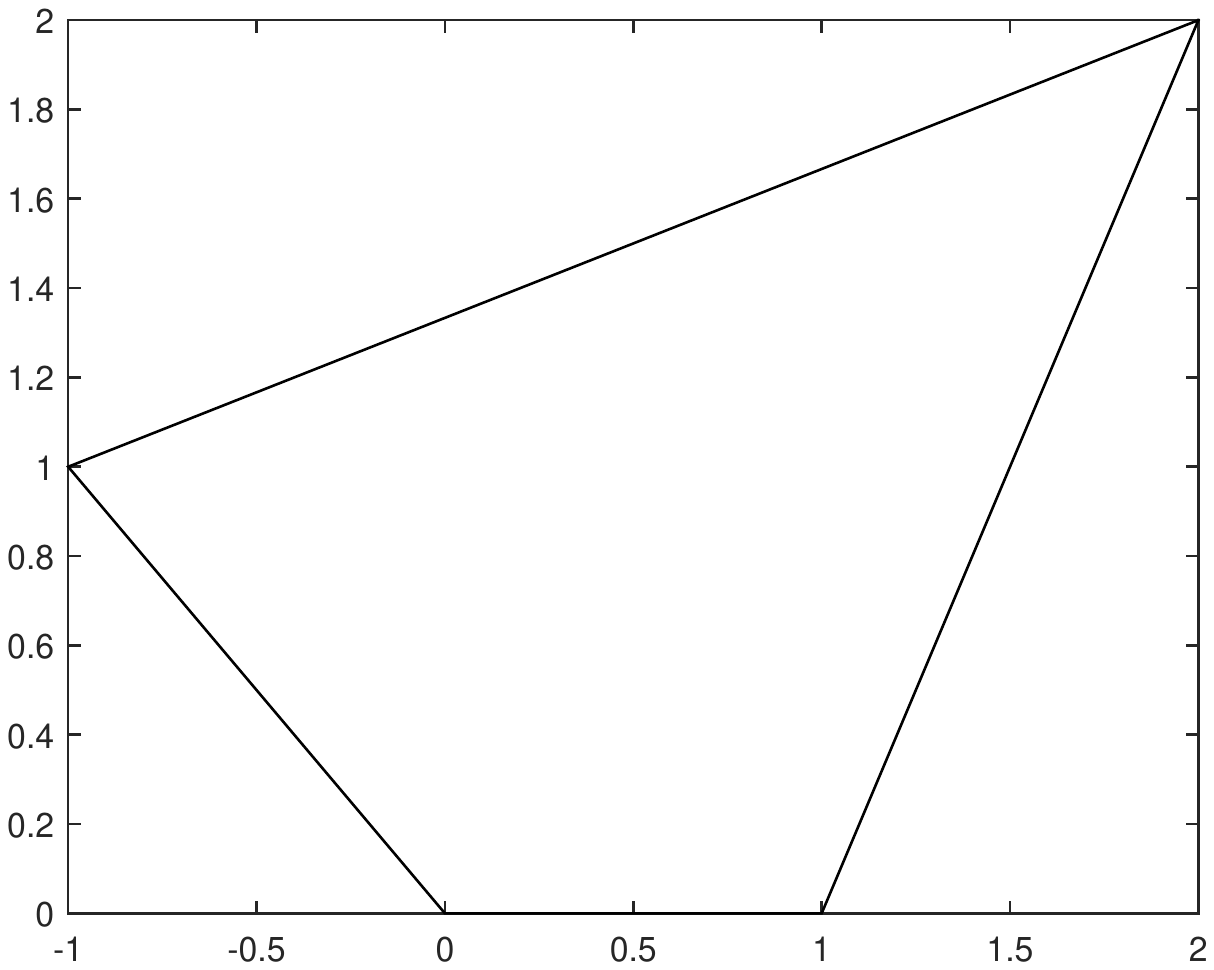}\\
			\vspace{-3.8cm}
			\includegraphics[width=1.651in]{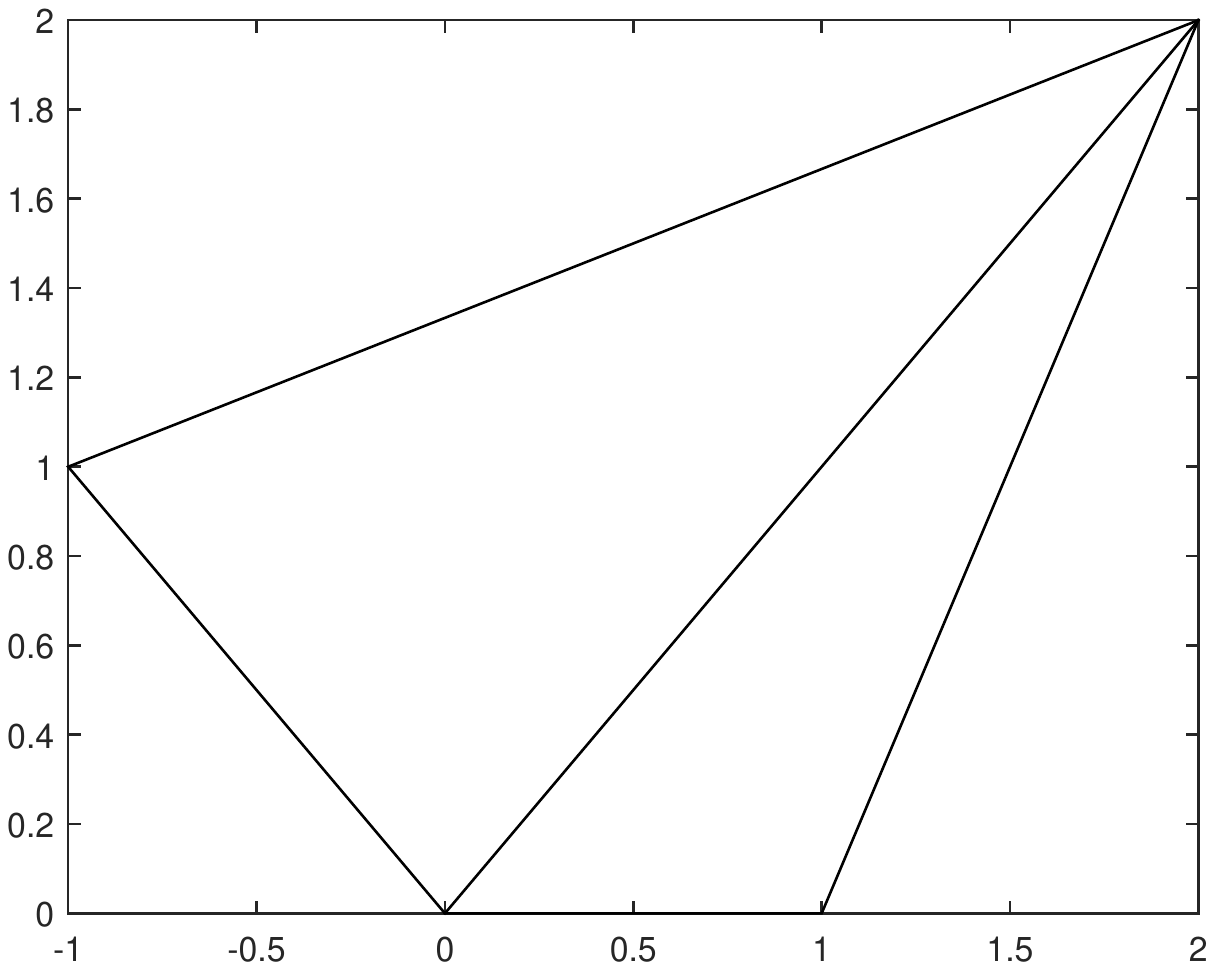}\\
			\vspace{-1.8cm}
		\end{minipage}
	}
	\subfigure[]{
		\begin{minipage}[t]{0.23\linewidth}
			\centering
			\includegraphics[width=1.651in]{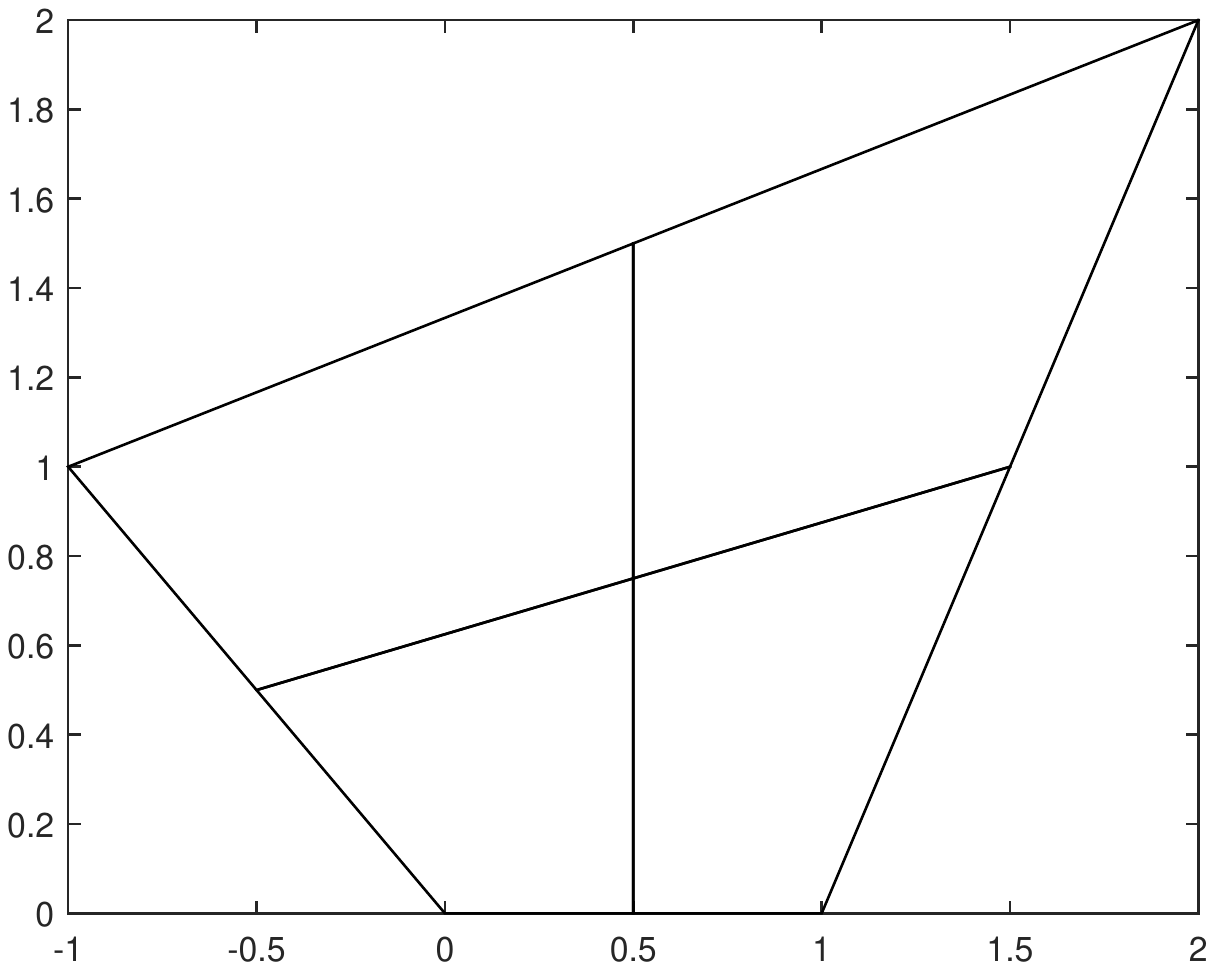}\\
			\vspace{-3.8cm}
			\includegraphics[width=1.651in]{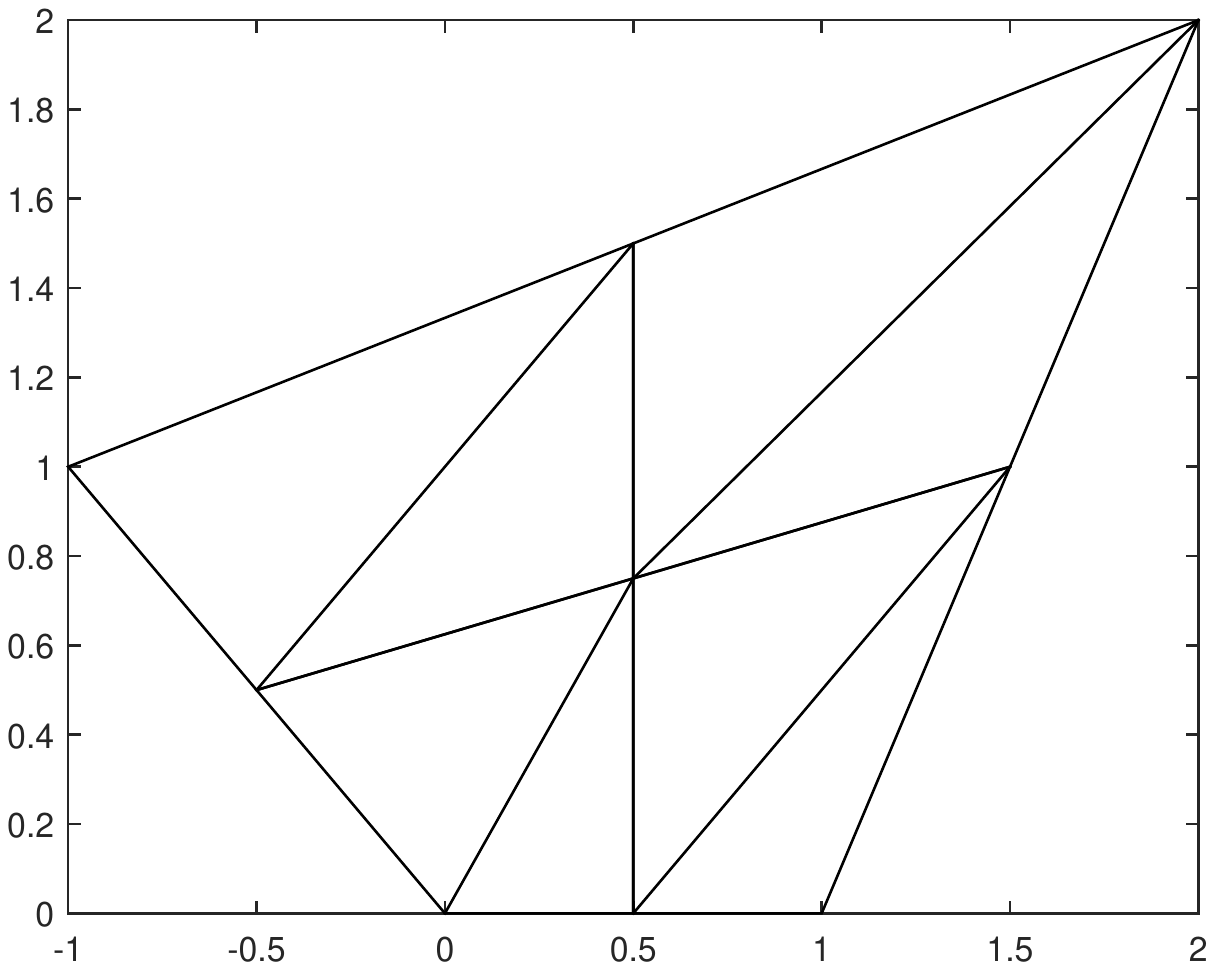}\\
			\vspace{-1.8cm}
		\end{minipage}
	}
	\subfigure[]{
		\begin{minipage}[t]{0.23\linewidth}
			\centering
			\includegraphics[width=1.651in]{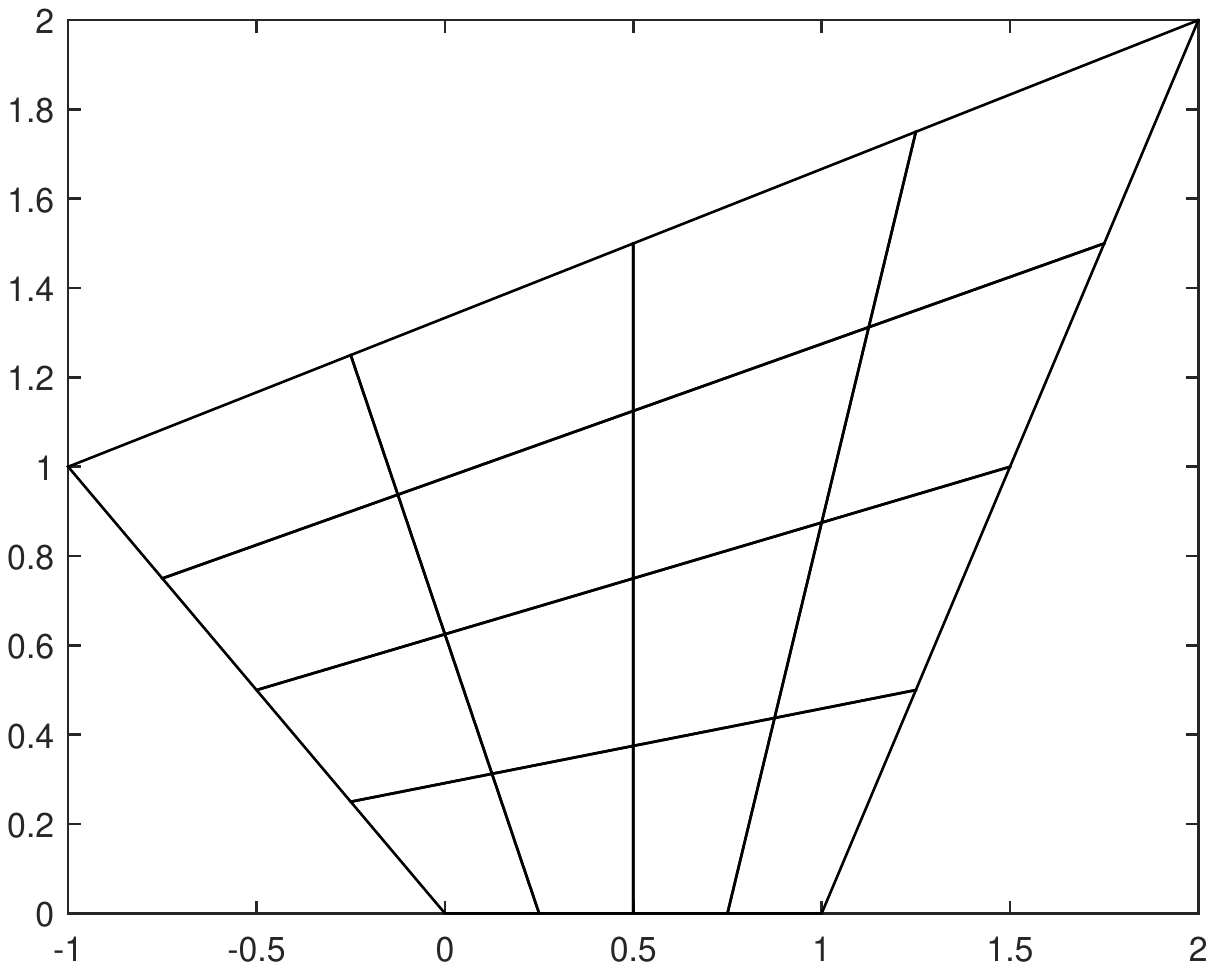}\\
			\vspace{-3.8cm}
			\includegraphics[width=1.651in]{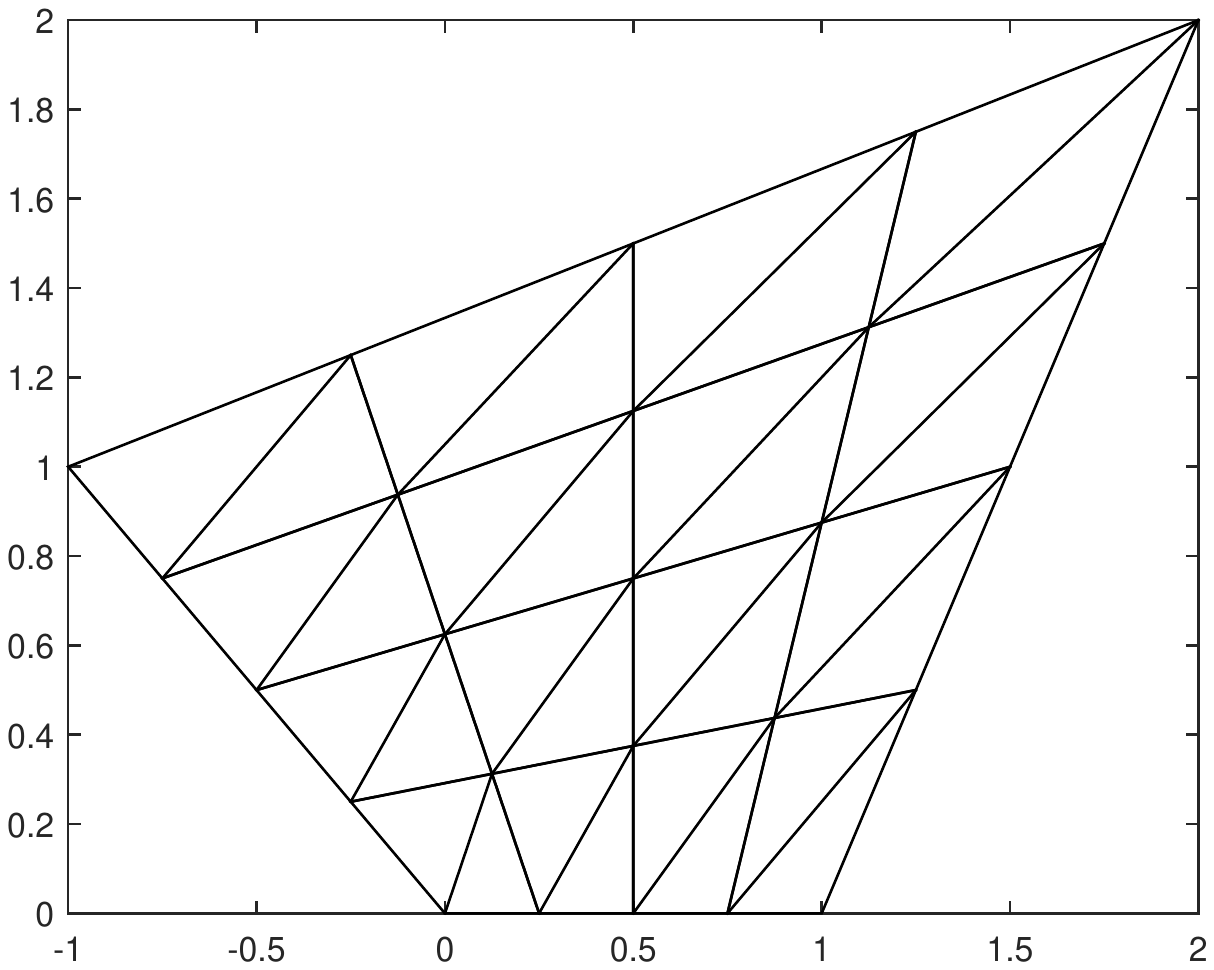}\\
			\vspace{-1.8cm}
		\end{minipage}
	}
	\subfigure[]{
		\begin{minipage}[t]{0.23\linewidth}
			\centering
			\includegraphics[width=1.651in]{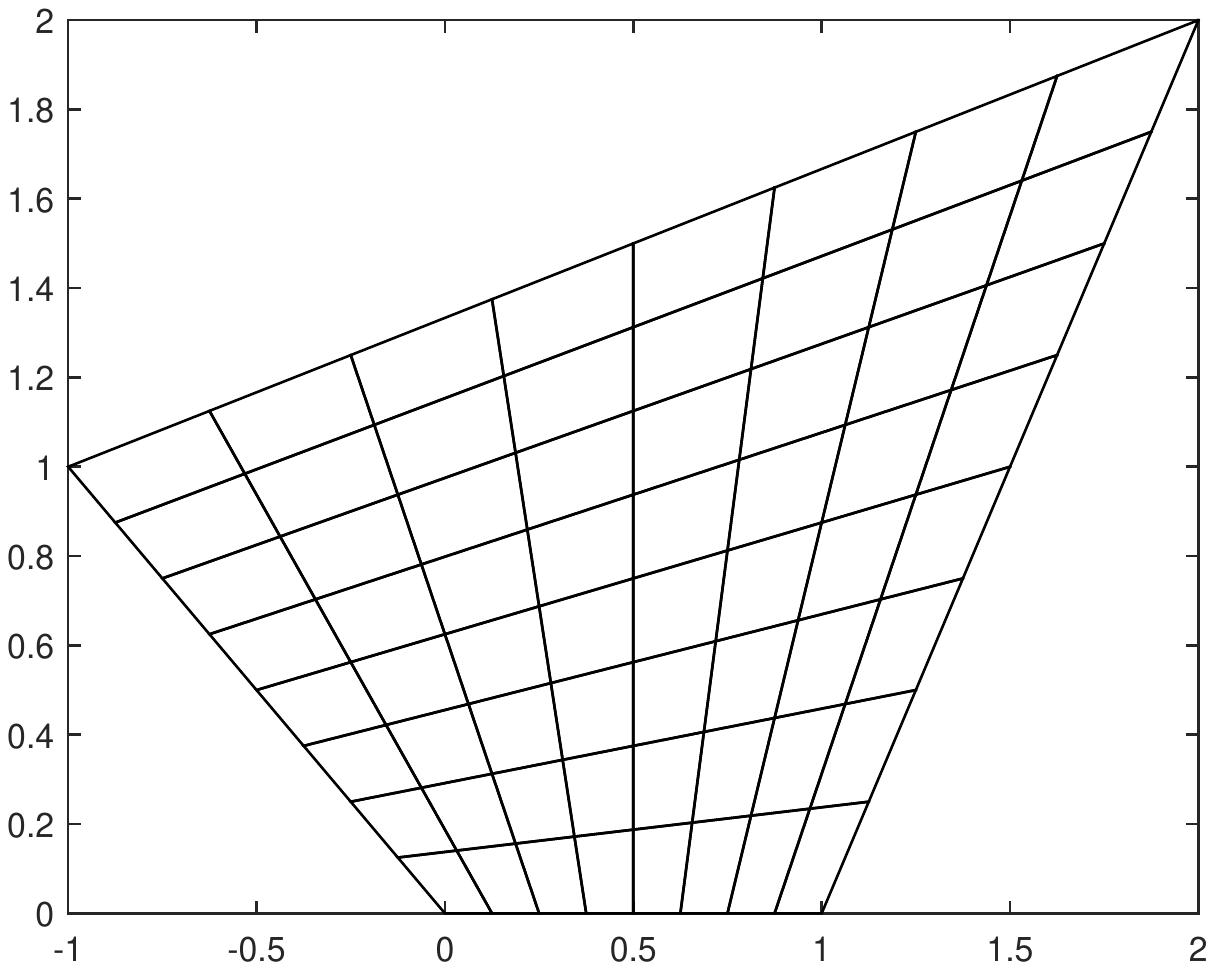}\\
			\vspace{-3.8cm}
			\includegraphics[width=1.651in]{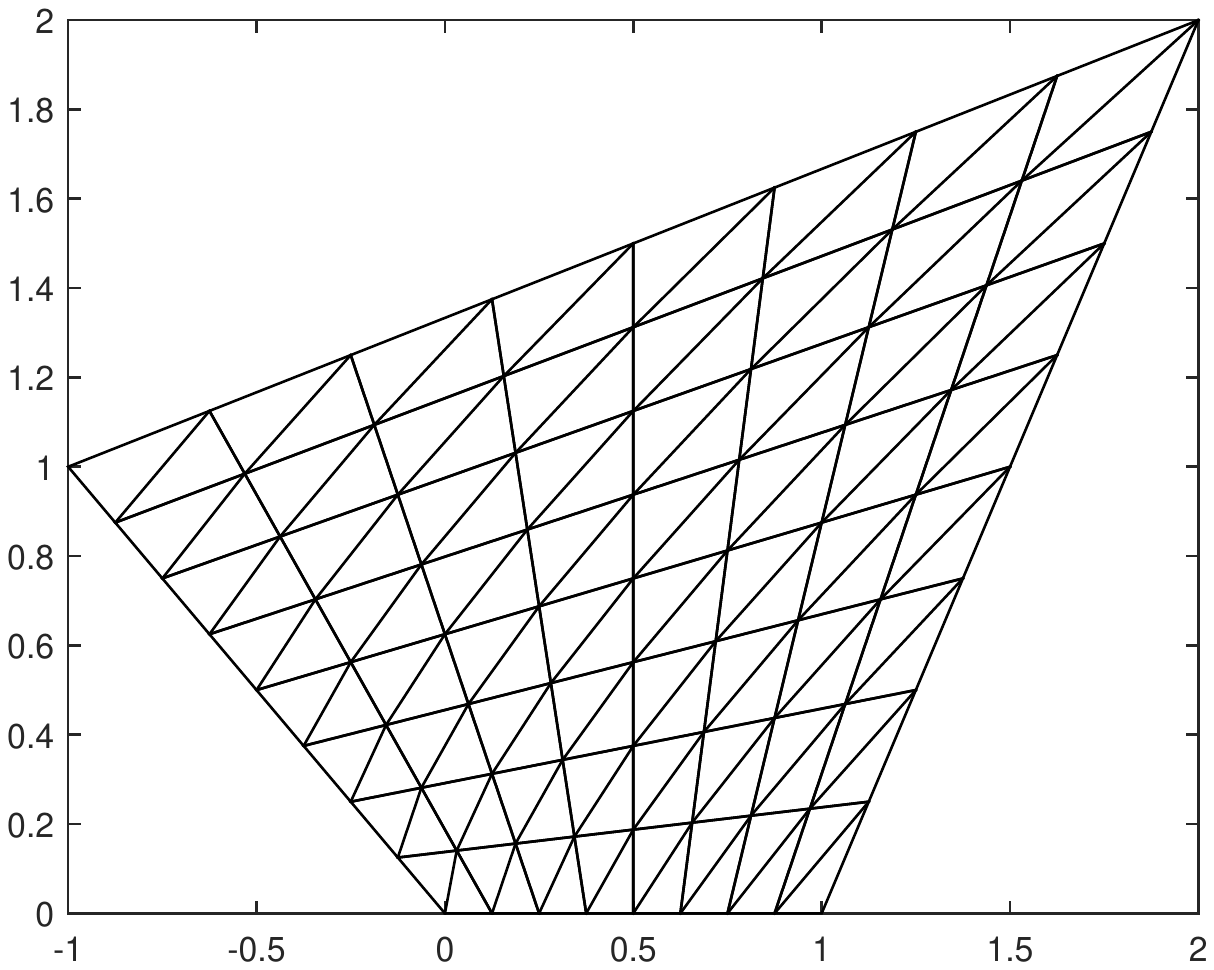}\\
			\vspace{-1.8cm}
		\end{minipage}
	}
	\centering
	\caption{Two different sequences of grids}
	\vspace{-0.2cm}
	\label{fig:Poissonshowmesh}
\end{figure}

\begin{table}[htbp]
	\caption{Numerical results for Poisson problem}\label{tab:tab6}
	\begin{center}
		\begin{tabular}{|c|ccccc|}
			\hline
			Size($\mathcal{J}_{h}$) & 
			\multicolumn{2}{c}{On quadrilateral grids} & 
			\quad & 
			\multicolumn{2}{c|}{On triangle grids} 
			\\ 
			\quad & 
			$|u-u_{h}|_{1,h}$ & 
			$\|u-u_{h}\|_{0,\Omega}$ & 
			\quad & 
			$|u-u_{h}|_{1,h}$ & 
			$\|u-u_{h}\|_{0,\Omega}$
			\\ \hline
			$8\times8$ & 
			1.67E0 & 
			1.39E-1 & 
			\quad & 
			3.59E0 & 
			2.57E-1
			\\
			$16\times16$ & 
			8.35E-1 & 
			3.52E-2 & 
			\quad & 
			1.83E0 & 
			6.73E-2
			\\
			$32\times32$ & 
			4.18E-1 & 
			9.69E-3 & 
			\quad & 
			9.23E-1 & 
			1.70E-2
			\\
			$64\times64$ & 
			2.09E-1 & 
			2.42E-3 & 
			\quad & 
			4.62E-1 & 
			4.57E-3
			\\ \hline
			Convergence order & 
			1 & 
			2 & 
			\quad & 
			1 & 
			2
			\\ \hline
		\end{tabular}
	\end{center}
\end{table}

Figure \ref{fig:convergence order on Poisson problem} 
reports on approximation results of ${\rm QBL}$ and Courant elements for Poisson equation.
The $x$-axis and the $y$-axis 
represent the logarithm of grid size $h$ and of the error, respectively. 
The dashed line and the solid line 
represent the error associated with the norm 
$|\cdot|_{1,h}$ and $\|\cdot\|_{0,\Omega}$, respectively.
The results confirm our conclusion: 
a clear first-order of convergence is observed with $|\cdot|_{1,h}$. 

\begin{figure}[htbp]
	\vspace{-7.5cm}
	\centering
	\includegraphics[width=1\linewidth]{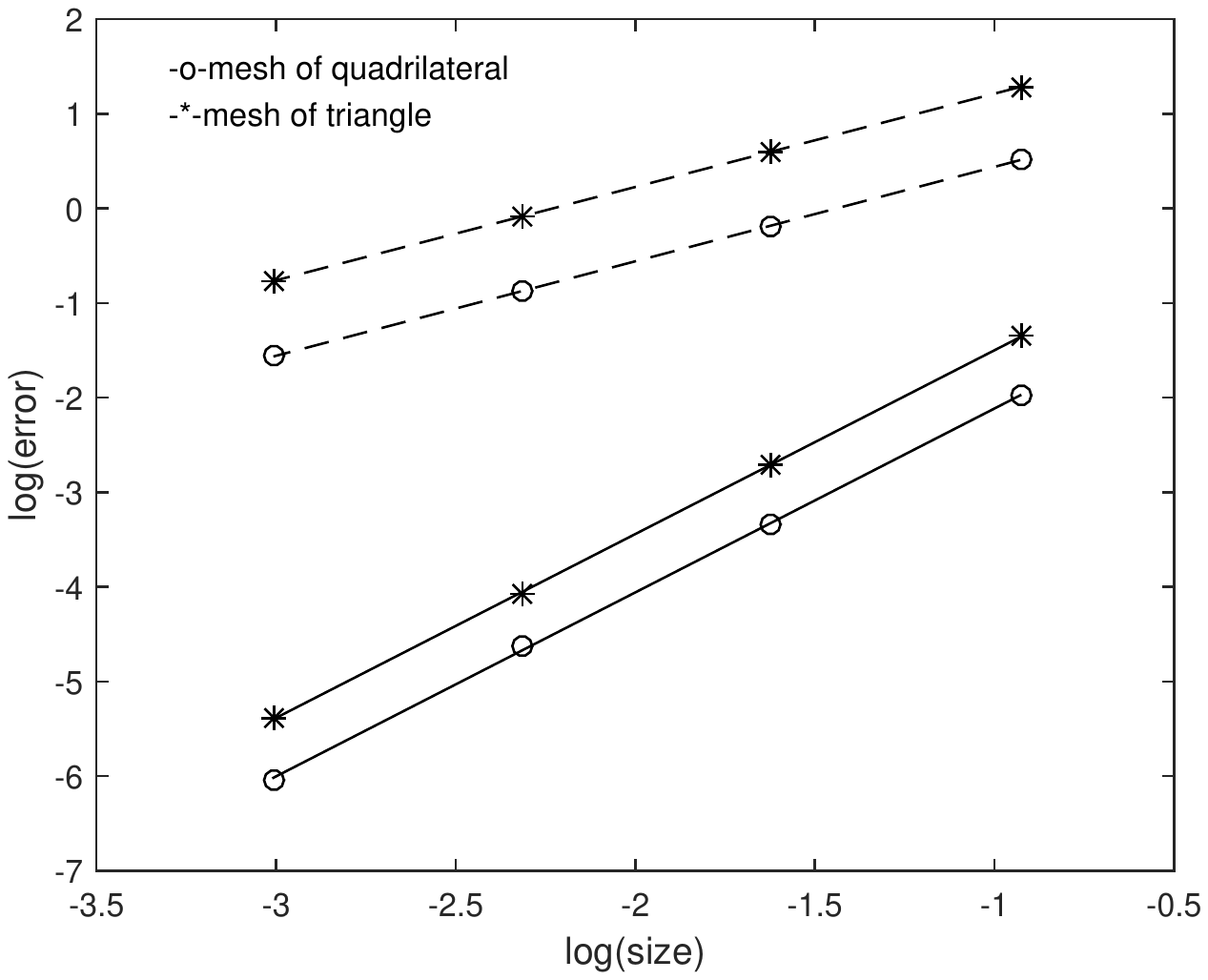}
	\centering
	\vspace{-8cm}
	\caption{The log-log plot of the error of ${\rm QBL}$ and Courant elements for Poisson equation}
	\label{fig:convergence order on Poisson problem}
\end{figure}

\subsection{Application on Laplace eigenvalue equation}
\label{subsec:application on Laplace eigenvalue equation}

We consider the Laplace eigenvalue problem with homogeneous boundary condition
\begin{equation}\label{eq:evpoisson}
\left\{
\begin{aligned}
-\Delta u & =\lambda u &  \quad & \mbox{in} \ \Omega \\
u & =0 &  \quad & \mbox{on} \ \Gamma.
\end{aligned}
\right.
\end{equation}

The variational formulation is to find $(\lambda,u)\in \mathbb{R}\times H^1_0(\Omega)$, such that 
\begin{equation}
\int_\Omega \nabla u\nabla v \ \mathrm{d}x=\lambda\int_\Omega uv \ \mathrm{d}x,\quad\forall\,v\in H^1_0(\Omega).
\label{eq:variational problem on evpoisson}
\end{equation}
The finite element problem is to find $(\lambda_h,u_h)\in \mathbb{R}\times V_{h0}^{\rm QBL}$, such that 
\begin{equation}
\sum_{K\in\mathcal{J}_h}\int_K\nabla u_h\nabla v_h \ \mathrm{d}x=\lambda_h\int_\Omega u_hv_h \ \mathrm{d}x,\quad\forall\,v_h\in V_{h0}^{\rm QBL}.
\label{eq:dicrete on evpoisson}
\end{equation}
\begin{theorem}
	Let the eigenvalues of the problem \eqref{eq:variational problem on evpoisson} and \eqref{eq:dicrete on evpoisson} be sorted from small to big. Let $(\lambda, u)$ and $(\lambda_h,u_h)$ be the $k$-th eigenpair of \eqref{eq:variational problem on evpoisson} and \eqref{eq:dicrete on evpoisson}, respectively. Then for $h$ small enough,
	\begin{equation}
	|\lambda-\lambda_h|=\mathcal{O}(h^2)\ \ \ \mbox{and}\ \ \ |u-u_h|_{1,h}=\mathcal{O}(h).
	\end{equation}
\end{theorem}
\begin{proof}
	The theorem is proved by the standard technique.
\end{proof}
\subsubsection{Numerical verification}
We choose the computational domain to be the unit square 
$\Omega = (0,1)\times(0,1)$. 
The eigenvalue $\lambda$ is chosen 
such that the exact solution is the function 
$u = \sin(\pi x)\sin(\pi y)$. 
We first divide the computational domain into four trapezoids, 
then use the same strategy as Subsection \ref{subsec:a FE scheme for the Poisson equation} to generate the grids, see Figure \ref{fig:Laplaeigshowmesh}, and repeat numerical test by same elements as before. The results are recorded in Table \ref{tab:tab7}.

\begin{figure}[htbp]
	\vspace{-2.1cm}
	\centering
	\subfigure[]{
		\begin{minipage}[t]{0.23\linewidth}
			\centering
			\includegraphics[width=1.651in]{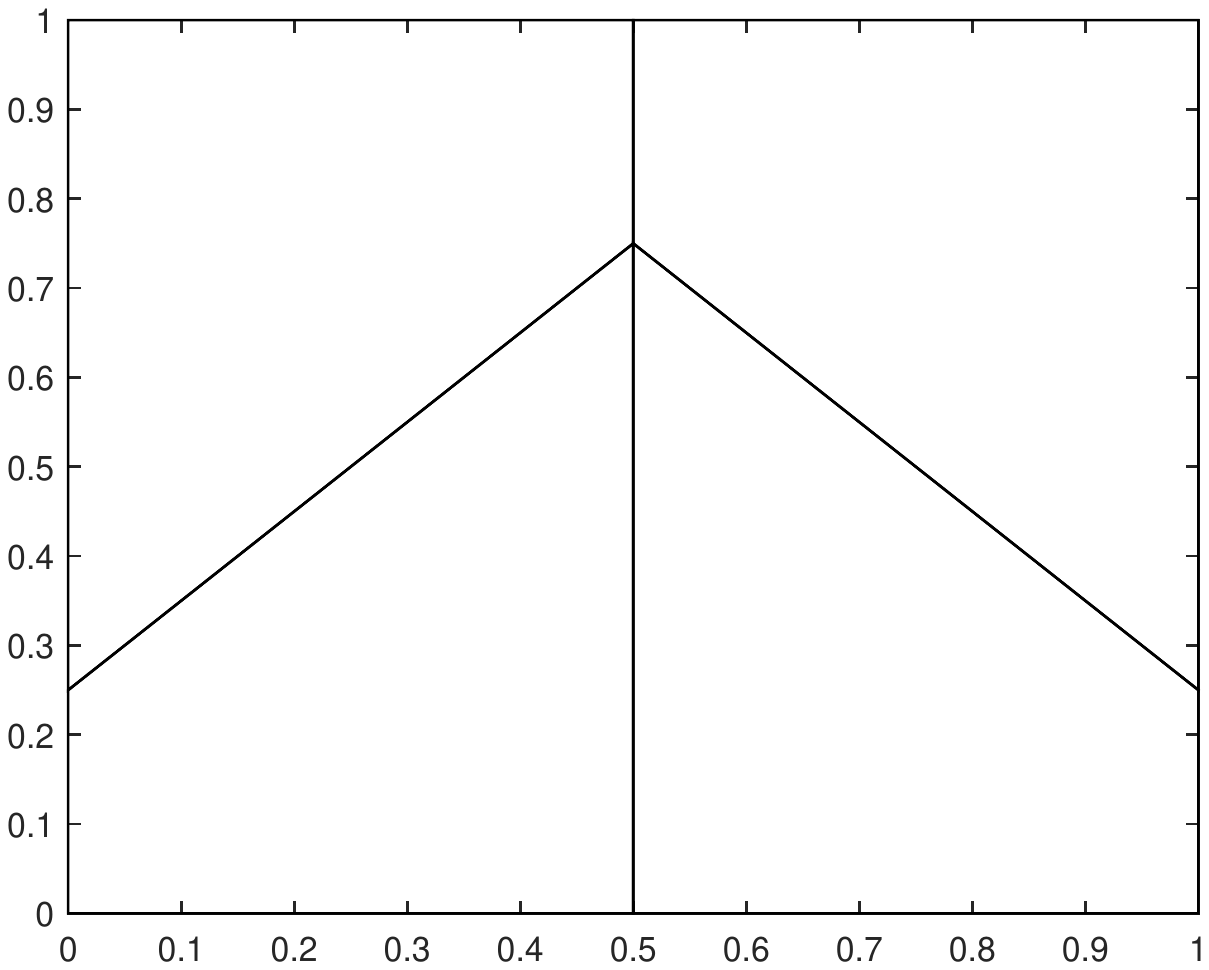}\\
			\vspace{-3.8cm}
			\includegraphics[width=1.651in]{Laplaeigmesh_1.pdf}\\
			\vspace{-1.8cm}
		\end{minipage}
	}
	\subfigure[]{
		\begin{minipage}[t]{0.23\linewidth}
			\centering
			\includegraphics[width=1.651in]{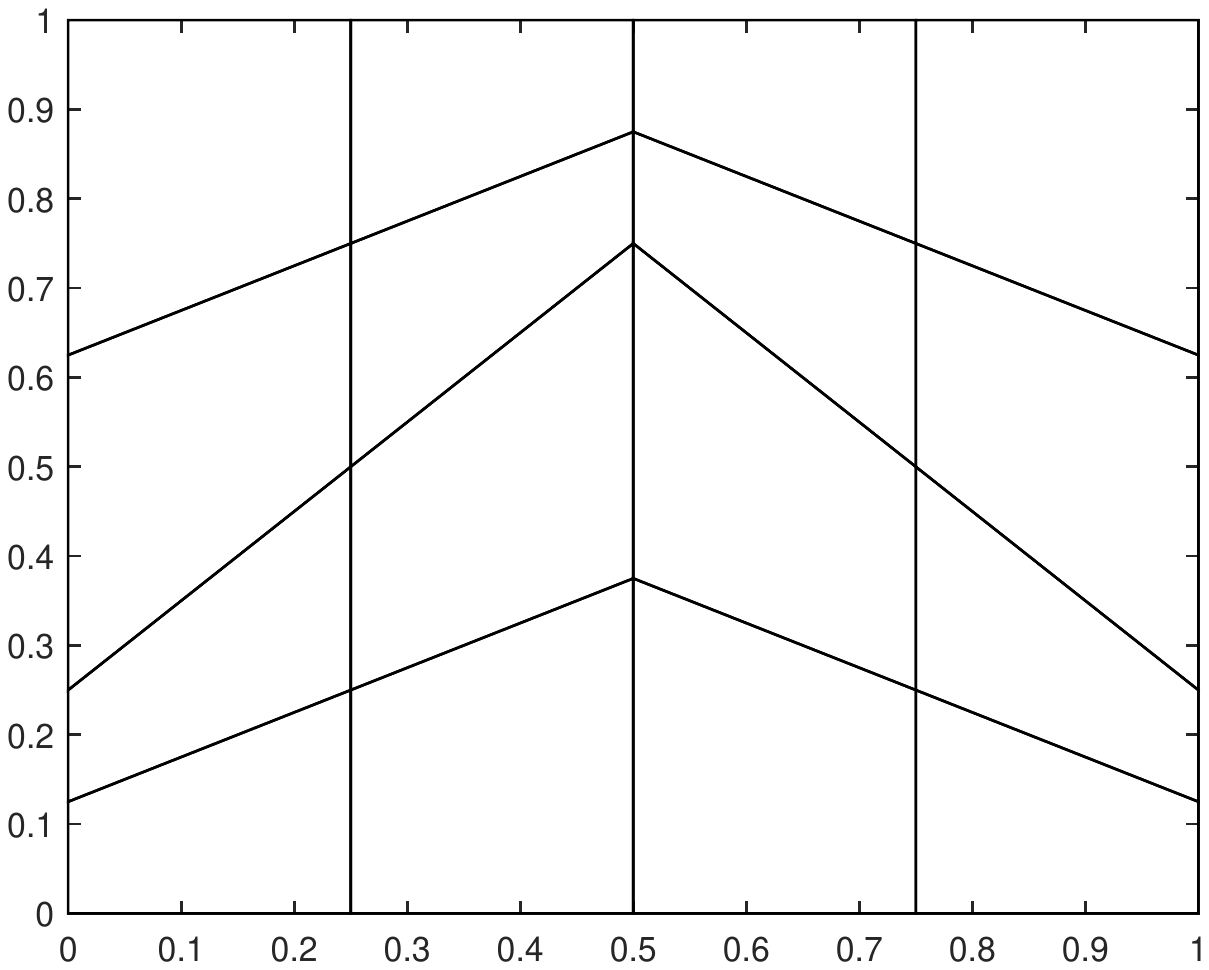}\\
			\vspace{-3.8cm}
			\includegraphics[width=1.651in]{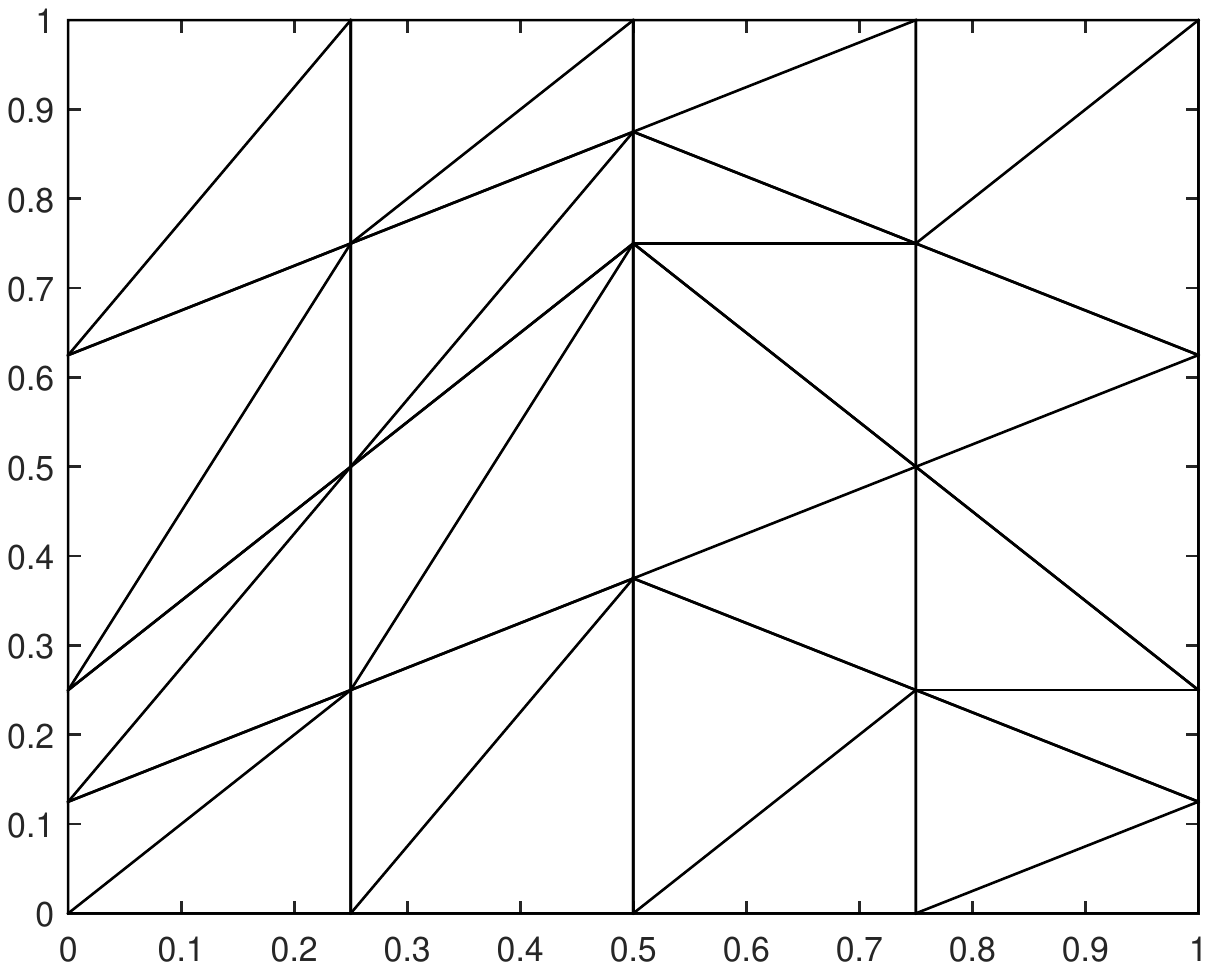}\\
			\vspace{-1.8cm}
		\end{minipage}
	}
	\subfigure[]{
		\begin{minipage}[t]{0.23\linewidth}
			\centering
			\includegraphics[width=1.651in]{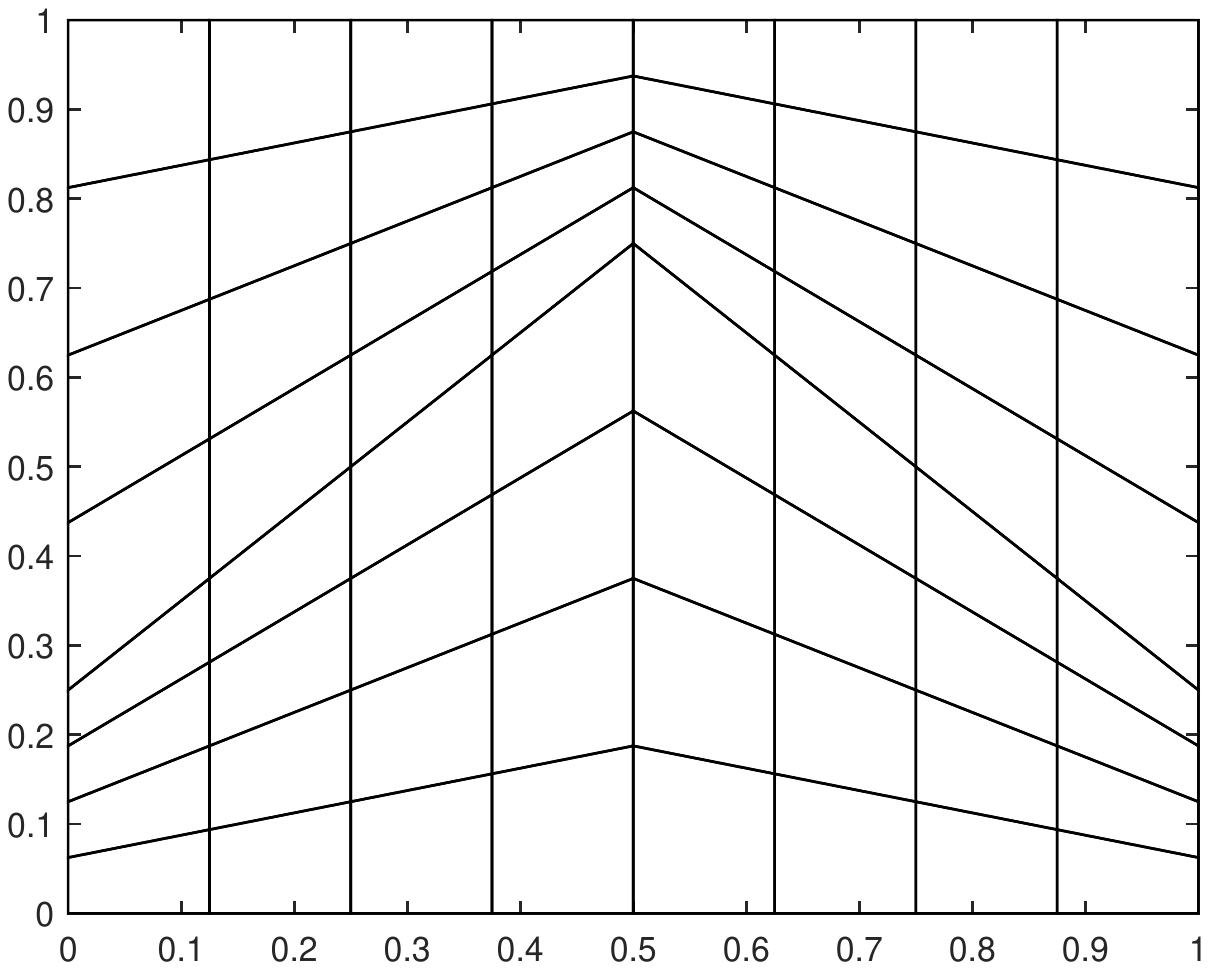}\\
			\vspace{-3.8cm}
			\includegraphics[width=1.651in]{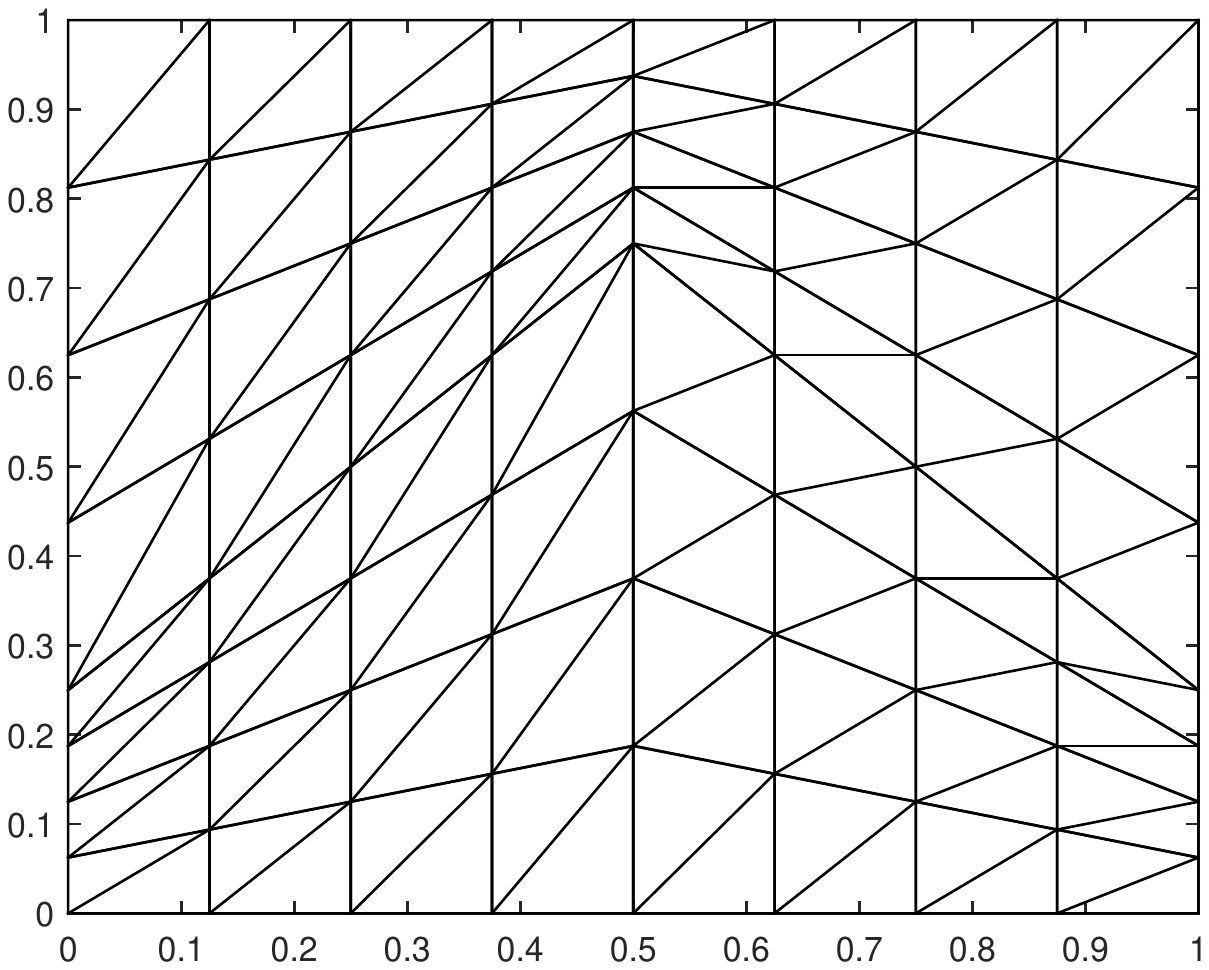}\\
			\vspace{-1.8cm}
		\end{minipage}
	}
	\subfigure[]{
		\begin{minipage}[t]{0.23\linewidth}
			\centering
			\includegraphics[width=1.651in]{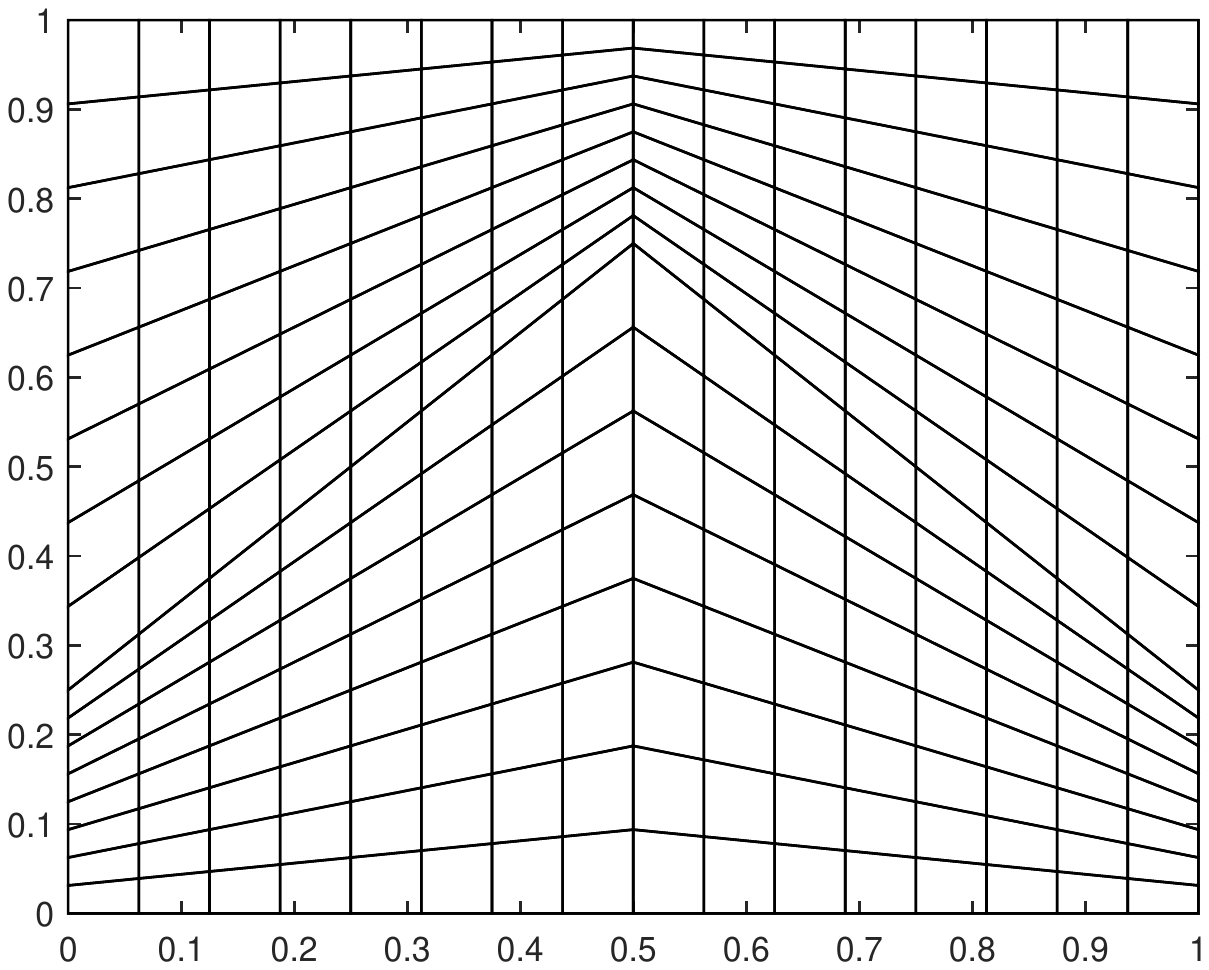}\\
			\vspace{-3.8cm}
			\includegraphics[width=1.651in]{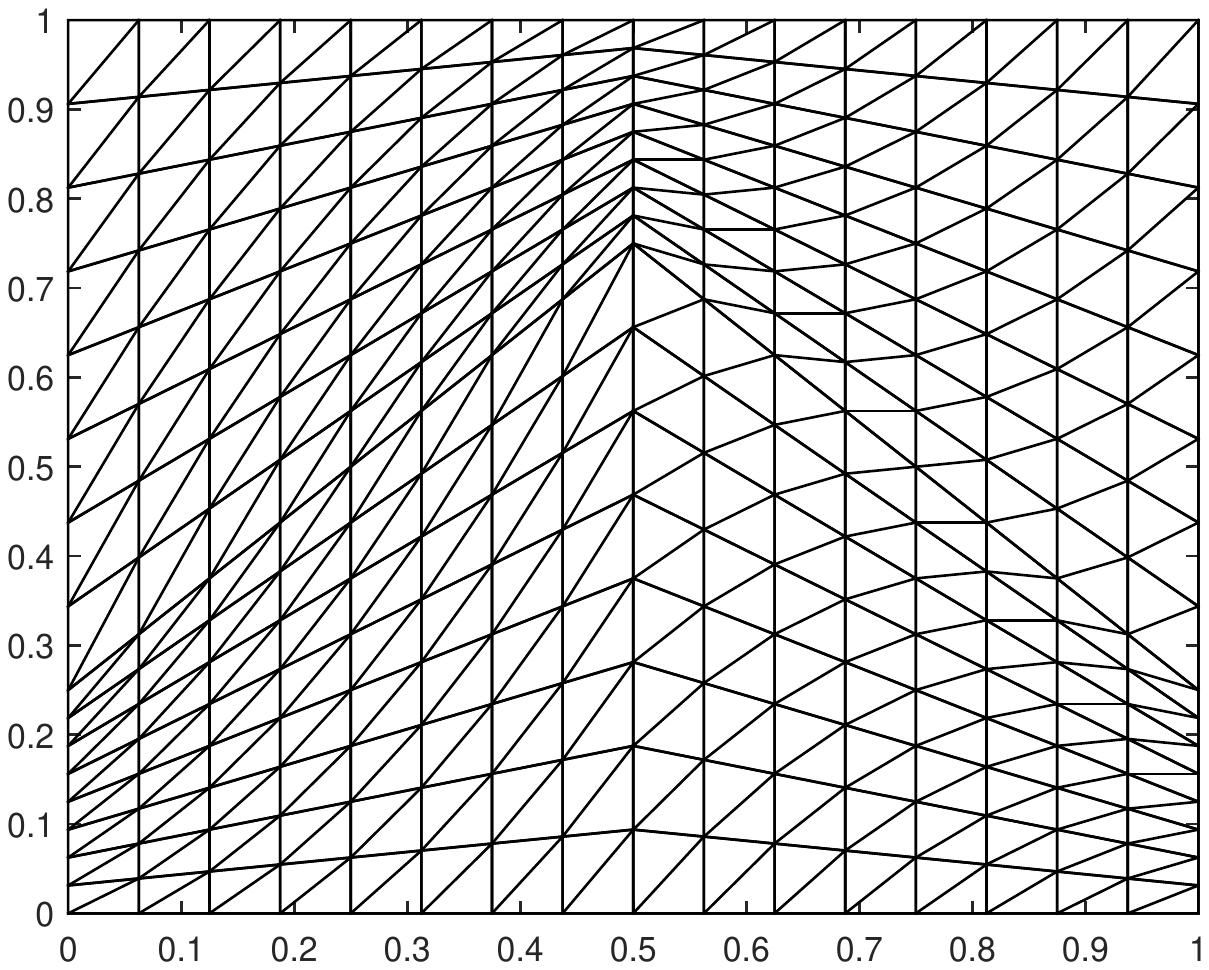}\\
			\vspace{-1.8cm}
		\end{minipage}
	}
	\centering
	\caption{Two different sequences of grids}
	\vspace{-0.3cm}
	\label{fig:Laplaeigshowmesh}
\end{figure}

\begin{table}[htbp]
	\caption{Numerical results for Laplace eigenvalue problem}\label{tab:tab7}
	\begin{center}
		\begin{tabular}{|c|cccccc|}
			\hline
			Size($\mathcal{J}_{h}$) & 
			\multicolumn{2}{c}{On quadrilateral grids} & 
			\quad & 
			\multicolumn{2}{c}{On triangle grids} &
			\quad
			\\ 
			\quad & 
			$\lambda_{h}$ & 
			$|\lambda-\lambda_{h}|$ & 
			\quad & 
			$\lambda_{h}$ & 
			$|\lambda-\lambda_{h}|$ &
			$\lambda$
			\\ \hline
			$8\times8$ & 
			2.035E1 & 
			6.090E-1 & 
			\quad & 
			2.074E1 & 
			9.995E-1 &
			2$\pi^2$
			\\
			$16\times16$ & 
			1.988E1 & 
			1.359E-1 & 
			\quad & 
			1.998E1 & 
			2.424E-1 &
			2$\pi^2$
			\\
			$32\times32$ & 
			1.977E1 & 
			3.210E-2 & 
			\quad & 
			1.980E1 & 
			6.120E-2 &
			2$\pi^2$
			\\
			$64\times64$ & 
			1.9747E1 & 
			7.800E-3 & 
			\quad & 
			1.9754E-1 & 
			1.480E-2 &
			2$\pi^2$
			\\ \hline
			Convergence order & 
			\quad & 
			2 & 
			\quad & 
			\quad & 
			2 &
			\quad
			\\ \hline
		\end{tabular}
	\end{center}
\end{table}

Figure \ref{fig:convergence order on Laplace eigenvalue problem}
reports on approximation results of ${\rm QBL}$ and Courant elements for Laplace eigenvalue equation.
The $x$-axis and the $y$-axis 
represent the logarithm of grid size $h$ and of the error, respectively. 
In this numerical experiment, a clear second-order of convergence is observed and the numerical performance of ${\rm QBL}$ element is better than that of Courant element.
\begin{figure}[htbp]
	\vspace{-7.2cm}
	\centering
	\includegraphics[width=1\linewidth]{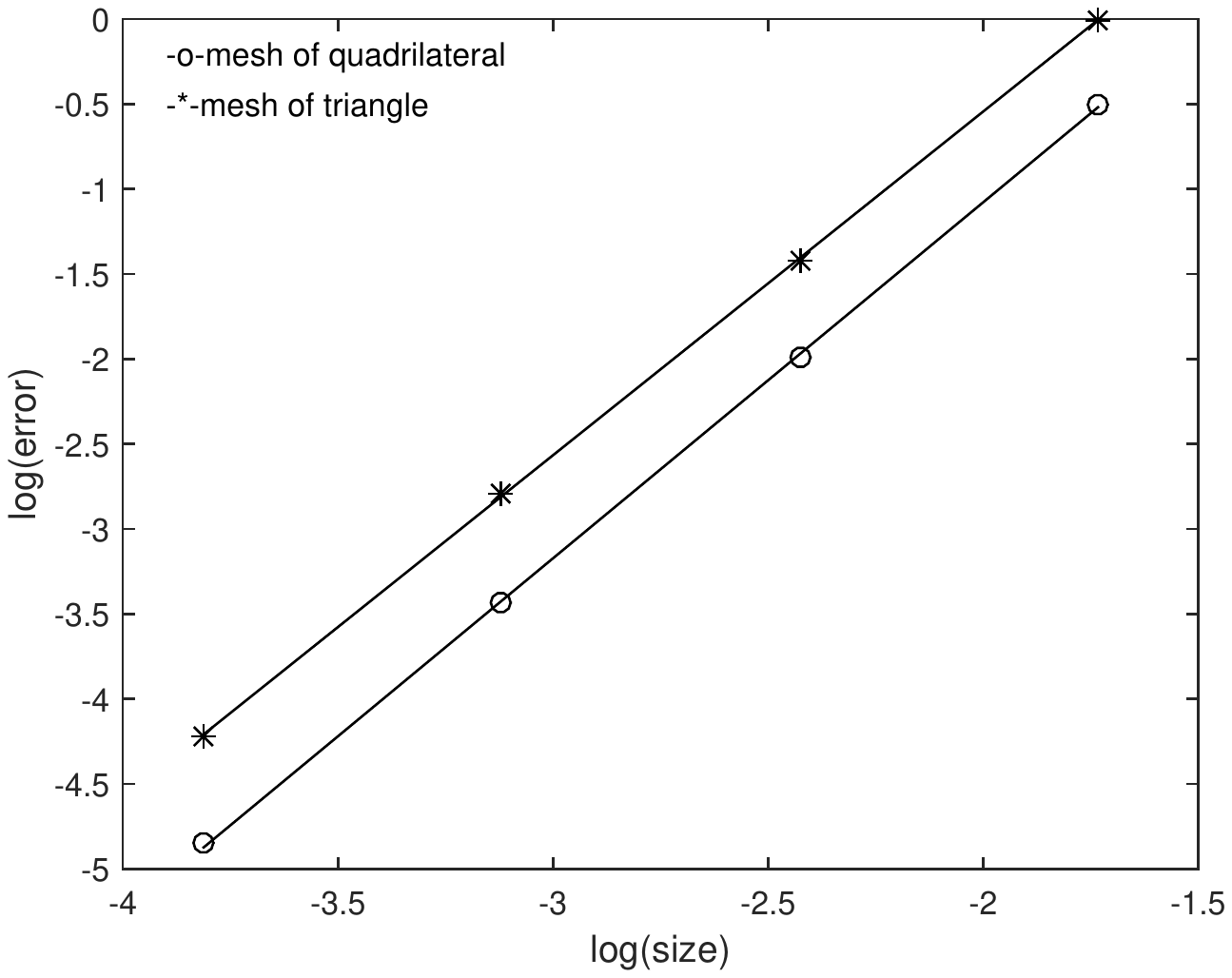}
	\centering
	\vspace{-8cm}
	\caption{The log-log plot of the error of ${\rm QBL}$ and Courant elements for Laplace eigenvalue equation}
	\label{fig:convergence order on Laplace eigenvalue problem}
\end{figure}

\subsection{Application on $H(\mathrm{rot})$ equation}
\label{subsec:application on rot equation}

We consider the problem with homogeneous boundary condition

\begin{equation*}
\left\{
\begin{aligned}
\undertilde{\mathrm{curl}}
&(\mathrm{rot}\usigma)+\usigma=\uf &  
\quad & \mbox{in} \ \Omega \\
&\usigma\ut=0 &  
\quad & \mbox{on} \ \Gamma.
\end{aligned}
\right.
\end{equation*}

The variational formulation is to find 
$\usigma\in H_0({\rm rot},\Omega)$, such that 
\begin{equation}
\int_\Omega {\rm rot}\usigma{\rm rot}\utau+\usigma\utau \ \mathrm{d}x=\int_\Omega \uf\utau\ \mathrm{d}x,\quad\forall\,v\in H_0({\rm rot},\Omega).
\label{equ:variational problem on rot}
\end{equation}
The finite element problem is to find $\usigma{}_h\in V_{h0}^{\rm QRT}$, such that 
\begin{equation}
\sum_{K\in\mathcal{J}_h}\int_K{\rm rot}\usigma{}_h{\rm rot}\utau{}_h+\usigma{}_h\utau{}_h\ \mathrm{d}x=\int_\Omega \uf\utau{}_h \ \mathrm{d}x,\quad\forall\,\utau{}_h\in V_{h0}^{\rm QRT}.
\label{equ:discrete problem on rot}
\end{equation}
\begin{theorem}
	Let $\usigma\in H^1({\rm rot},\Omega)\cap H_0({\rm rot},\Omega)$ and $\usigma{}_h$ be the solutions of \eqref{equ:variational problem on rot}, and \eqref{equ:discrete problem on rot}, respectively. Then
	\begin{equation}
	\|\usigma-\usigma{}_h\|_{{\rm rot},h}\leqslant Ch(|\usigma|_{1,\Omega}+|{\rm rot}\usigma|_{1,\Omega}).
	\end{equation}
\end{theorem}
\begin{proof}
	The theorem is proved by the standard technique.
\end{proof}
\subsubsection{Numerical verification}
We choose the computational domain to be the unit square 
$\Omega = (0,1)\times(0,1)$. 
The source term $\uf$ is chosen such that the exact solution is given by $\usigma=(xy^2-xy,x^2y-xy)^{T}$. Then we test the performance of the ${\rm QRT}$ element on the quadrilateral grids as Subsection \ref{subsec:application on Laplace eigenvalue equation} and the results are recorded in Table \ref{tab:tab8}.

\begin{table}[htbp]
	\caption{Numerical results for $H(\mathrm{rot})$ problem}\label{tab:tab8}
	\begin{center}
		\begin{tabular}{|c|ccc|}
			\hline
			Size($\mathcal{J}_{h}$) & 
			$\|\usigma-\usigma{}_{h}\|_{0,\Omega}$ &
			$|\usigma-\usigma{}_{h}|_{\mathrm{rot},h}$ &
			$\|\usigma-\usigma{}_{h}\|_{\mathrm{rot},h}$
			\\ \hline
			$8\times8$ & 
			1.13E-1 &
			5.95E-2 &
			1.28E-1
			\\
			$16\times16$ & 
			5.53E-2 & 
			2.98E-2 & 
			6.28E-2 
			\\
			$32\times32$ & 
			2.78E-2 & 
			1.49E-2 & 
			3.15E-2
			\\
			$64\times64$ & 
			1.40E-2 & 
			7.45E-3 & 
			1.58E-2 
			\\ \hline
			Convergence order & 
			1 & 
			1 &
			1 
			\\ \hline
		\end{tabular}
	\end{center}
\end{table}

Figure \ref{fig:convergence order on rot problem}
reports on approximation results of ${\rm QRT}$  element for $H(\mathrm{rot})$ equation.
The $x$-axis and the $y$-axis 
represent the logarithm of grid size $h$ and of the error, respectively. 
The error associated with $\|\cdot\|_{0,\Omega}$, 
$|\cdot|_{\mathrm{rot},h}$ and $\|\cdot\|_{\mathrm{rot},h}$
are plotted by dotted line, dashed line and solid line, respectively.
The results confirm our conclusion: a clear first-order of convergence is observed.

\begin{figure}[htbp]
	\vspace{-7.4cm}
	\centering
	\includegraphics[width=1\linewidth]{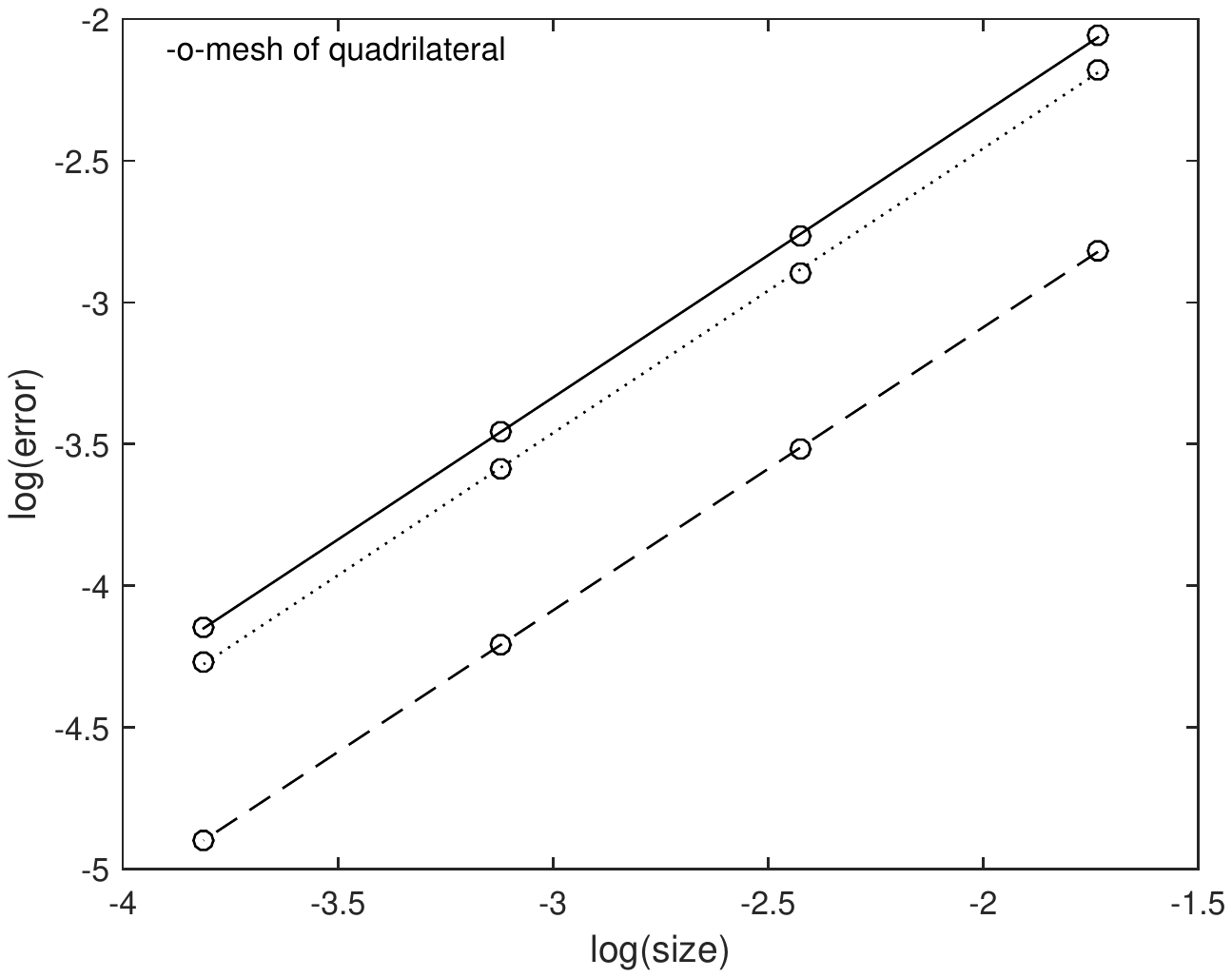}
	\centering
	\vspace{-8cm}
	\caption{The log-log plot of the error of ${\rm QRT}$ for $H(\mathrm{rot})$ equation}
	\label{fig:convergence order on rot problem}
\end{figure}

\section{Concluding remarks}
\label{sec:concluding remarks}
In this paper, we present a polynomial de Rham complex on a grid that consists of arbitrary quadrilaterals by constructing two nonconforming finite elements for the $H^1$ and $H(\rot)$ problems, respectively. As is proved in the present paper, cf. Theorem \ref{thm:nocon} and Remark \ref{rem:nocon}, we can not theoretically expect the finite element be conforming anyway; however, the two spaces are both quasi-conforming and are consistent to the requirement of the differential complex. Moreover, with respect to the $\mathcal{O}(h^2)$ asymptotic parallelogram assumption, the scheme for $H(\rot)$ problem is $\mathcal{O}(h)$ convergent for $H^1(\rot)$ exact solution; namely, this element does not suffer from the extra requirement on the regularity for general nonconforming $H(\curl)$ element (cf. \cite{shi2009low}). For the Poisson equation, numerical experiments show that the {\rm QBL} element plays superior to the triangular linear element with the same amount of unknowns for both source and eigenvalue problems, which confirms the need of quadrilateral elements.
~\\

It follows immediately that a piecewise polynomial complex can be constructed according to
\begin{equation}
\mathbb{R}\xrightarrow{\rm inclusion} H^1  \xrightarrow{\rm curl}  H(\dv) \xrightarrow{\dv} L^2
\end{equation}
by simply a rotation. Further, the methodology which uses the same shape functions as that on a framework of parallelogram and the nodal parameters on a physical cells can be generalized to more complicated cases, such as higher order schemes and higher dimension problems. These would be discussed in future.

\bibliographystyle{plain}
\bibliography{reference}

\begin{thebibliography}{10}

\bibitem{arbogast2016two}
Todd Arbogast and Maicon~R Correa.
\newblock Two families of h (div) mixed finite elements on quadrilaterals of
  minimal dimension.
\newblock {\em SIAM Journal on Numerical Analysis}, 54(6):3332--3356, 2016.

\bibitem{arbogast2018direct}
Todd Arbogast and Zhen Tao.
\newblock Direct serendipity and mixed finite elements on convex
  quadrilaterals.
\newblock {\em arXiv preprint arXiv:1809.02192}, 2018.

\bibitem{arnold2002approximation}
Douglas Arnold, Daniele Boffi, and Richard Falk.
\newblock Approximation by quadrilateral finite elements.
\newblock {\em Mathematics of computation}, 71(239):909--922, 2002.

\bibitem{arnold2005quadrilateral}
Douglas~N Arnold, Daniele Boffi, and Richard~S Falk.
\newblock Quadrilateral h (div) finite elements.
\newblock {\em SIAM Journal on Numerical Analysis}, 42(6):2429--2451, 2005.

\bibitem{bochev2008rehabilitation}
Pavel Bochev and Denis Ridzal.
\newblock Rehabilitation of the lowest-order raviart-thomas element on
  quadrilateral grids.
\newblock {\em SIAM Journal on Numerical Analysis}, 47(1):487--507, 2008.

\bibitem{brenner2007mathematical}
Susanne Brenner and Ridgway Scott.
\newblock {\em The mathematical theory of finite element methods}, volume~15.
\newblock Springer Science \& Business Media, 2007.

\bibitem{carstensen2012explicit}
Carsten Carstensen, Joscha Gedicke, and Donsub Rim.
\newblock Explicit error estimates for courant, crouzeix-raviart and
  raviart-thomas finite element methods.
\newblock {\em Journal of Computational Mathematics}, pages 337--353, 2012.

\bibitem{ciarlet2002finite}
Philippe~G Ciarlet.
\newblock {\em The finite element method for elliptic problems}, volume~40.
\newblock Siam, 2002.

\bibitem{dubach2009pseudo}
Eric Dubach, Robert Luce, and Jean-Marie Thomas.
\newblock Pseudo-conforming polynomial finite elements on quadrilaterals.
\newblock {\em International Journal of Computer Mathematics},
  86(10-11):1798--1816, 2009.

\bibitem{falk2011hexahedral}
Richard~S Falk, Paolo Gatto, and Peter Monk.
\newblock Hexahedral h (div) and h (curl) finite elements.
\newblock {\em ESAIM: Mathematical Modelling and Numerical Analysis},
  45(1):115--143, 2011.

\bibitem{jeon2013class}
Youngmok Jeon, Hyun Nam, Dongwoo Sheen, and Kwangshin Shim.
\newblock A class of nonparametric dssy nonconforming quadrilateral
  elements∗.
\newblock {\em ESAIM: Mathematical Modelling and Numerical Analysis},
  47(6):1783--1796, 2013.

\bibitem{park2003p}
Chunjae Park and Dongwoo Sheen.
\newblock P 1-nonconforming quadrilateral finite element methods for
  second-order elliptic problems.
\newblock {\em SIAM Journal on Numerical Analysis}, 41(2):624--640, 2003.

\bibitem{payne1960optimal}
Lawrence~E Payne and Hans~F Weinberger.
\newblock An optimal poincar{\'e} inequality for convex domains.
\newblock {\em Archive for Rational Mechanics and Analysis}, 5(1):286--292,
  1960.

\bibitem{rannacher1992simple}
Rolf Rannacher and Stefan Turek.
\newblock Simple nonconforming quadrilateral stokes element.
\newblock {\em Numerical Methods for Partial Differential Equations},
  8(2):97--111, 1992.

\bibitem{shi2009low}
Dongyang Shi and Lifang Pei.
\newblock Low order crouzeix--raviart type nonconforming finite element methods
  for the 3d time-dependent maxwell’s equations.
\newblock {\em Applied Mathematics and Computation}, 211(1):1--9, 2009.

\bibitem{shi1984convergence}
Zhong-Ci Shi.
\newblock A convergence condition for the quadrilateral wilson element.
\newblock {\em Numerische mathematik}, 44(3):349--361, 1984.

\bibitem{zeng2019optimal}
Huilan Zeng, Chensong Zhang, and Shuo Zhang.
\newblock Optimal quadratic element on rectangular grids for $h^1$ problems.
\newblock {\em arXiv: Numerical Analysis 1903.00938}, 2019.

\bibitem{zhang2004nested}
Shangyou Zhang.
\newblock On the nested refinement of quadrilateral and hexahedral finite
  elements and the affine approximation.
\newblock {\em Numerische Mathematik}, 98(3):559--579, 2004.

\bibitem{zhang2016stable}
Shuo Zhang.
\newblock Stable finite element pair for stokes problem and discrete stokes
  complex on quadrilateral grids.
\newblock {\em Numerische Mathematik}, 133(2):371--408, 2016.

\bibitem{zhang2004polynomial}
Zhimin Zhang.
\newblock Polynomial preserving gradient recovery and a posteriori estimate for
  bilinear element on irregular quadrilaterals.
\newblock {\em Int. J. Numer. Anal. Model}, 1(1):1--24, 2004.

\end{thebibliography}

\end{document}